\documentclass[a4paper]{amsart}

\usepackage{amsmath,amsthm,amsfonts,latexsym,amssymb,amscd,enumerate,times}
\usepackage[all]{xy}
\usepackage{color}
\usepackage[backref]{hyperref}
\usepackage{graphicx}

\theoremstyle{theorem}
\newtheorem{theorem}{Theorem}[section]
\newtheorem{thm}[theorem]{Theorem}
\newtheorem{prop}[theorem]{Proposition}
\newtheorem{lem}[theorem]{Lemma}
\newtheorem{cor}[theorem]{Corollary}

\newtheorem{add}[theorem]{Addendum}

\theoremstyle{definition}
\newtheorem{defn}[theorem]{Definition}
\newtheorem{definition}[theorem]{Definition}
\newtheorem{rem}[theorem]{Remark}

\newtheorem{corollary}[theorem]{Corollary}

\theoremstyle{definition}

\theoremstyle{remark}

\DeclareMathOperator{\ind}{ind}

\newcommand{\naturals}{\mathbb{N}}
\newcommand{\integers}{\mathbb{Z}}
\newcommand{\rationals}{\mathbb{Q}}

\newcommand{\boundary}{\partial}

\newcommand{\iso}{\cong}
\newcommand{\Riem}{\operatorname{Riem}}

\newcommand{\Diff}{\operatorname{Diff}}

\newcommand{\Top}{\operatorname{Top}}

\DeclareMathOperator{\Homeo}{Homeo} 
\DeclareMathOperator{\Nil}{{\Nil}^3}

\DeclareMathOperator{\Spin}{Spin}

\newcommand{\Z}{\mathbb{ Z}}

\newcommand{\Q}{\mathbb{ Q}}
\newcommand{\R}{\mathbb{ R}}

\newcommand{\id}{\operatorname{id}}

{\catcode`@=11\global\let\c@equation=\c@theorem}

\begin{document}

\title{The space of metrics of positive scalar curvature}

\author{Bernhard Hanke \and Thomas Schick \and
  Wolfgang Steimle}\thanks{Thomas Schick was supported by the Courant
    Center ``Higher order structures'' of the institutional strategy of  the 
    Universit\"at G\"ottingen within the German excellence initiative. Wolfgang Steimle was supported by the Hausdoff Center for Mathematics, Bonn. }
\address{Institut f\"ur Mathematik, Universit\"at Augsburg}
\address{Mathematisches Institut, Georg-August-Universit\"at G\"ottingen}
\address{Mathematisches Institut, Universit\"at Bonn}
\email{hanke@math.uni-augsburg.de}
\email{schick@uni-math.gwdg.de}
\email{steimle@math.uni-bonn.de}

\begin{abstract} We study the topology of the space of positive scalar curvature metrics on high dimensional spheres and
  other spin manifolds. Our main result provides elements in
  higher homotopy and homology groups of these spaces, which, in contrast to
previous approaches, are of infinite order and survive in the (observer) moduli 
space of such metrics. 

Along the way we construct smooth fiber bundles over spheres whose total spaces have 
non-vanishing $\hat{A}$-genera, thus establishing  the non-multiplicativity of  
the $\hat{A}$-genus in fiber bundles with simply connected base.

\end{abstract}

\maketitle

%%%%%%%%%%%%%%%%%%%%%%%%%%%%%%%%%%%%%%%%%%%%%%%%%%%%%%%%%%%%%%%%%%%%%%%%%%%%%%%%%%%%%
%%%%%%%%%%%%%%%%%%%%%%%%%%%%%%%%%%%%%%%%%%%%%%%%%%%%%%%%%%%%%%%%%%%%%%%%%%%%%%%%%%%%%

\section{Introduction and summary}

The classification of positive scalar curvature metrics on closed smooth
manifolds is
a central topic in Riemannian geometry. Whereas the existence question has  
been resolved in many cases and is governed by the (stable) Gromov-Lawson
conjecture (compare e.g.~\cite{StolzICM,Schick}), information on the
topological complexity of the 
space of positive scalar curvature metrics on a given manifold $M$ has been sparse and only
recently some progress has been made \cite{BHSW, CS}.

We denote by ${\Riem}^+(M)$ the space of Riemannian metrics of positive 
scalar curvature, equipped with the $C^{\infty}$-topology, on a closed smooth manifold $M$. 
If it is not empty, we want to give information on the homotopy groups 
 $\pi_k({\Riem}^+(M), g_0 )$ for $g_0 \in {\Riem}^+(M)$ in the different
 path components of ${\Riem}^+(M)$. One method 
 to construct non-zero elements in these homotopy groups, developed by
 Hitchin,  is to pull back $g_0$ along a family of diffeomorphisms
 of $M$.  In \cite[Theorem 4.7]{Hitchin} this was used to prove existence of 
 non-zero classes of order two in $\pi_1({\Riem}^+(M),g_0)$ for certain manifolds $M$. 
 In \cite[Corollary 1.5]{CS} this method has been refined to show
 that there exist non-zero elements of 
 order two in infinitely many degrees of $\pi_*({\Riem}^+(M),g_0)$, when $M$
 is a spin manifold admitting a metric $g_0$ of positive scalar curvature. 
 
In our paper  we construct non-zero elements of \emph{infinite order} in  
 $\pi_k ({\Riem}^+(M), g_0)$ for $k\in\naturals$,  among others.   Our construction is
 quite different from Hitchin's in that
it does not rely on topological properties of the diffeomorphism group of
$M$.

Our first main result reads as follows. 

\begin{theorem} \label{main_general} Let $k \geq 0$ be a natural number. Then
  there  is a natural number $N(k)$ with the following 
properties: 
\begin{itemize} 
   \item[a)]  For each $n \geq N(k)$ and each spin manifold $M$ admitting a metric
$g_0$ of positive scalar curvature and of dimension $4n-k-1$,  the homotopy
group
\[
    \pi_{k} ({\Riem}^+(M), g_0) 
\]
contains elements of infinite order if $k\ge 1$, and infinitely many different
elements if $k=0$. Their images under the Hurewicz homomorphism in
$H_k(\Riem^+(M))$ still have infinite order.
\item[b)]  For $M=S^{4n-k-1}$, the images of these elements in the homotopy and homology groups of
the observer moduli space $\Riem^+(S^{4n-k-1}) / \Diff_{x_0}(S^{4n-k-1})$, see Definition \ref{def:observer}, have infinite order. 
\end{itemize} 
\end{theorem}

For $k=0$ these statements are well known with $N(k) =2$ and $\Diff(S^{4n-1})$ instead of $\Diff_{x_0}(S^{4n-1})$, see \cite[Theorem
IV.7.7]{LawsonMichelsohn} and \cite[Theorem 4.47]{GLIHES}. 

Refined versions of part b) of Theorem \ref{main_general} will be stated in Theorems \ref{theo:main_geom} 
and \ref{main_moduli} below. 

Following a suggestion by one of the referees we remark that Theorem \ref{main_general} implies the following 
stable statement for manifolds in arbitrary dimension.

{\begin{prop}
  For each $k\geq 1$ and each spin manifold $M$ in dimension $(3-k)\;{\rm
    mod}\;4$ and admitting a metric of positive scalar curvature, the
  homotopy group $\pi_k({\Riem}^+(M \times B^{N} ))$ contains elements of
  infinite order. Here, $N$ is a constant depending on $k$ and $\dim M$,  $B$
  denotes a  Bott manifold, an eight dimensional closed simply
  connected spin manifold satisfying $\hat{A}(B) = 1$, and $B^{N}$ is the
  $N$-fold cartesian product.
\end{prop}}
Note that, as a spin manifold with non-vanishing $\hat{A}$-genus,
the manifold $B$ does not admit a metric of positive scalar
curvature. Inspired by the stable Gromov-Lawson-Rosenberg conjecture, we wonder
whether there is a stability pattern in $\pi_k(\Riem^+(M\times B^N))$ for
growing $N$.

Remarkably, at the end of \cite[Section 5]{GLIHES} Gromov and Lawson write:  ``\textit{The
construction above can be greatly generalized using  
Browder-Novikov Theory. A similar construction detecting higher homotopy
groups of the space of positive scalar  
curvature metrics can also be made.}'' We are not sure what
Gromov and Lawson precisely had in mind here. Our paper may be viewed
as an attempt to realize the program hinted at by these remarks. 

Recall that the diffeomorphism group $\Diff(M)$ acts on $\Riem^+(M)$ via pull-back of metrics. 
The orbit space $\Riem^+(M)/\Diff(M)$ is the {\em moduli space} of Riemannian metrics of positive scalar 
curvature. Because the action is not free (the isotropy group at $g \in \Riem^+(M)$ is equal to the 
isometry group of $g$, which is a compact Lie group by the Myers-Steenrod theorem) it is convenient to restrict the action 
to the following subgroup of $\Diff(M)$.

\begin{definition}\label{def:observer}
 Let $M$ be a connected smooth manifold and let $x_0 \in M$. The
  \emph{diffeomorphism group with observer} $\Diff_{x_0}(M)$ is the subgroup of
  $\Diff(M)$ consisting of diffeomorphisms $\phi\colon M \to M$ fixing $x_0
  \in M$ and with $D_{x_0}\phi=\id_{T_{x_0}M}$. 

This definition first appeared in \cite{AkuBot}. 
It is easy to see that $\Diff_{x_0}(M)$, unlike $\Diff(M)$, 
acts freely on the space of Riemannian metrics on $M$ (as long as $M$
is connected). The orbit space ${\Riem}^+(M) / \Diff_{x_0}(M)$ 
is called the \emph{observer moduli space} of positive scalar curvature
metrics. 
\end{definition}

The construction of Theorem \ref{main_general} is based on the following fundamental result, which we 
regard  of independent interest.  

\begin{theorem} \label{technical} Given $k,l \geq 0$  there is an
  $N=N(k,l)\in\naturals_{\ge 0}$ with the following property: For all $n\geq
  N$, there is a $4n$-dimensional smooth closed spin manifold $P$ with 
  non-vanishing $\hat A$-genus and which fits into a smooth fiber
  bundle
\[
    F \hookrightarrow P \to S^{k}.
\]
In addition we can assume that 
the following conditions are satisfied: \begin{enumerate} 
   \item\label{item:high_connect} The fiber $F$ is $l$-connected, and
   \item\label{item:section} the bundle $P \to S^{k}$ has a smooth section
     $s\colon  S^{k} \to P$ with trivial normal bundle. 
\end{enumerate}
\end{theorem}

Recall  that in a fiber bundle $F\to E\to B$
of oriented closed smooth manifolds the Hirzebruch $L$-genus is multiplicative if $B$ is simply connected, i.e.~$L(E) = L(B) \cdot L(F)$, see \cite{CHS}.
Theorem \ref{technical} shows that a corresponding statement 
for the $\hat{A}$-genus is wrong, even if $B$ is a sphere.  

We remark that our construction, which is based on abstract existence results
in differential topology,  does not yield an explicit description of  the
diffeomorphism type of the fiber manifold  $F$. 

\begin{rem} \label{alpha} In Theorem \ref{technical} we can assume in addition that $\alpha(F)=0$ by taking the fiber connected sum 
of $P$ with the trivial bundle $S^k \times (-F) \to S^k$ where $-F$ denotes the manifold $F$ with reversed spin structure (so 
that $-F$ represents the negative of $F$ in the spin bordism group). This can be 
done along a trivialized normal bundle of a smooth section $s \colon S^k \to P$ as in point (\ref{item:section})  of Theorem 
\ref{technical}. 

In this case  the fiber $F$ carries a metric of positive scalar curvature
by a classical result of Stolz \cite{StolzAnn} if $l \geq  1$ and $\dim F \geq
5$. 
But note that there is no global choice of such metrics on the fibers of $E \to S^k$ depending smoothly on the 
base point, 
 because this would imply that $P$ carries a metric of positive scalar curvature by considering a metric 
 on $P$ whose restriction to the vertical tangent space coincides with the given family of 
 positive scalar curvature metrics, shrinking the fibers and applying the O'Neill formulas \cite[Chaper 9.D.]{Besse}. 
 However, this is impossible as $P$ is spin and satisfies $\hat{A}(P) \neq 0$. 
\end{rem}

Let us point out an implication pertaining to diffeomorphism groups. For an oriented closed smooth 
manifold $F$ and a pointed 
continuous map $\phi \colon (S^{k},*) \to (\Diff^+(F), \id)$ into the group of orientation preserving diffeomorphisms of $F$ we define
 \[
       M_\phi = {\left( S^{k} \times F \times [0,1] \right)   \big/ \, (x,f,0)\sim(\psi(x,f), 1 )} 
 \]
 where {$\psi \colon S^k \times F \to S^k \times F$} is a smooth map homotopic to the adjoint of $\phi$ and  inducing diffeomorphisms $\psi(x,-) \colon F \to F$ for each $x\in S^k$ 
(depending on $\phi$ there is a canonical isotopy class of such 
 maps $\psi$). 
 We can regard $M_{\phi}$ as the total space of a fiber bundle $F \to M_{\phi} \to S^k \times S^1$ obtained by a parametrized
 mapping cylinder construction for $\psi(x,-)$. 
The diffeomorphism type of this bundle 
 depends only on the homotopy class of $\phi$. 
   
For each $k \geq 0$ this leads to a group homomorphism 
\[
\hat{A}_{{\rm Diff}} \colon \pi_{k}({\Diff}^+(F), \id)\to \rationals,  \quad [\phi]\mapsto \hat{A}(M_\phi). 
\]
If $F$ admits a spin structure and $k\ge 2$, this invariant takes values in
$\integers$, as in this case the manifold $M_\phi$ carries a spin structure, 
{$\psi\colon S^k\times F\to S^k\times F$ admitting a lift} to the spin principal bundle. 
By Proposition \ref{prop:Ahat_additivity} below the same conclusion holds {in general} if $F$ admits 
both a spin structure and a metric of positive scalar curvature.

\begin{cor}\label{corol:non-mult-Ahat}
 For each $k\ge 0$ there is an oriented closed smooth manifold $F$ such that
  the homomorphism $\hat{A}_{{\rm Diff}} \colon \pi_{k}({\Diff}^+(F), \id)\to \rationals$ 
  is   non-trivial. 
 \end{cor}
 
To our knowledge,  this is new for all $k \geq 0$.

  \begin{proof}
\begin{figure}\label{fig:cob1}
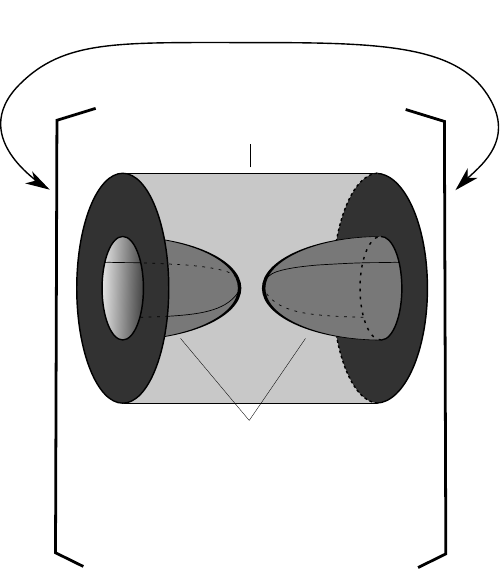
\caption{}
\end{figure}

Take a bundle $F \to P \to S^{k+1}$ from Theorem \ref{technical}. The total space can be obtained by a clutching 
  construction 
  \[
       P = (D^{k+1} \times F) \cup_{\psi} (D^{k+1} \times F)  
  \]
  for an appropriate smooth map $\psi \colon S^{k}\times F \to S^{k}\times F$. Denote by 
\[(W; S^k\times I, D^{k+1}\times\partial I)\subset D^{k+1}\times I\]
the standard bordism relative boundary of an index-zero surgery, as depicted in the upper part of 
Figure 1. Using $\psi$ to identify the left with the right end of $W\times F$ produces a bordism between $P$ and the required bundle over $S^{k}\times S^1$. The result follows from bordism invariance of the $\hat A$-genus.
 \end{proof}

The proof of Theorem \ref{technical} is based on results from classical
differential topology:
Surgery theory \cite{Davis}, Casson's theory of pre-fibrations \cite{Casson} and Hatcher's theory of concordance spaces \cite{Hatcher}. 
Theorem~\ref{main_general} for $M=S^{4n-k-1}$ follows from this with
the use of Igusa's fiberwise Morse theory \cite{Igusa2} and 
Walsh's generalization of the Gromov-Lawson surgery method to families of generalized Morse 
functions with critical points of {coindex} at least three \cite{Walsh}.

We will now formulate part b) of Theorem \ref{main_general} in a more general context. 

In order to express the fact that, contrary to the constructions in \cite{Hitchin} and \cite{CS},
our families of metrics in Theorem \ref{main_general} are not induced by families 
of diffeomorphisms of $M$, we propose the following definition.

\begin{definition}\label{def:geom_sign} 
Let $M$ be an oriented  closed smooth manifold and let $g_0$ be a positive scalar curvature metric on $M$. 
A class $c \in \pi_k ({\Riem}^+(M), g_0)$ is called {\em
    not geometrically significant}  if $c$ is represented by 
\[
  S^k\to \Riem^+(M), \quad t\mapsto \phi(t)^* g_0
\]
for some pointed continuous map $\phi\colon (S^k, *)\to (\Diff^+(M), \id)$.
Otherwise, $c$ is called \emph{geometrically significant}.
\end{definition}

Note that the homotopy classes of \cite{CS,Hitchin} are by their very
construction not geometrically significant.

\begin{definition} \label{def_multiplicative_fiber} 
Let $F$ be an oriented closed smooth manifold. We call $F$ \begin{itemize}
 \item an \emph{$ \hat A$-multiplicative fiber in 
  degree $k$} if for every oriented fiber bundle
  $F\to E\to S^{k+1}$ we have $\hat A(E)=0$.  (This implies that the map $\hat A\colon \pi_k({\Diff}^+(F), \id)\to \rationals$ defined 
  before Corollary \ref{corol:non-mult-Ahat} is zero). 
\item  a \emph{strongly $\hat 
    A$-multiplicative fiber in degree $k$} if for every oriented smooth fiber bundle
  $F\to
  E\to B$ over a closed oriented manifold $B$ of dimension
  $k+1$ we have $\hat A(E)= \hat{A}(B) \cdot \hat{A}(F)$.
 \end{itemize}
  \end{definition}
 
 Obviously $F$ is a strongly $\hat{A}$-multiplicative fiber
 in degree $k$ if $\dim F + k +1$ is not divisible by four.

\begin{prop} \label{prop_multiplicative_fiber}
\leavevmode
\begin{itemize} 
  \item[a)] If $F$ has vanishing rational Pontryagin classes, in particular if it is 
  stably parallelizable or a rational homology sphere, then $F$ is 
  an $\hat A$-multiplicative fiber in any degree. 
\item[b)] Every homotopy sphere is a strongly $\hat A$-multiplicative fiber in any degree.
\end{itemize} 
\end{prop} 

\begin{proof} 
For a) we observe that $F$ is framed in $E$, so that the rational Pontryagin classes $p_i(E)$ restrict to
$p_i(F)=0$. The long exact sequence of the pair $(E,F)$ implies that all $p_i(E)$
pull back from $H^*(E,F)$.

Denote by $D^{k+1}_\pm$ the upper and lower hemisphere of $S^{k+1}$, respectively. By
excision and the K\"unneth isomorphism we obtain isomorphisms of rational cohomology rings
\[ 
H^*(E,F)\cong H^*(E, E\vert_{D^{k+1}_+}) \cong H^*(E\vert_{D^{k+1}_-}, E\vert_{S^{k}}) \cong H^*(F)\otimes H^*(D^{k+1}_-, S^{k}).   
\]
But in $H^*(D^{k+1}_-, S^{k})$ all non-trivial products vanish,   and hence the same is true for $H^*(E,F)$.

So there are no non-trivial products of Pontryagin classes in $E$. It follows that the $\hat{A}$-genus of $E$ is a multiple 
of the signature of $E$, which is known to vanish in fiber bundles over spheres. 

For assertion b), we note that, given an  oriented bundle $E\to B$ with fiber $S^n$ and structure group the orientation
preserving homeomorphism
group of $S^n$, we obtain an oriented topological disc bundle $W\to B$ with $\boundary
W=E$ by  ``fiberwise coning off the spheres''. More precisely, the Alexander trick defines an embedding
$\Homeo^+(S^n)\to \Homeo^+(D^{n+1})$ by extending homeomorphisms $S^n \to S^n$ to homeomorphisms 
$D^{n+1} \to D^{n+1}$ radially over the cone over $S^{n}$, which is identified with $D^{n+1}$ in the standard way. 
This embedding splits the restriction homomorphism,
and $W$ is obtained by an associated bundle construction. 

Now \emph{rational} Pontryagin classes are defined for
topological manifolds and  the rational Pontryagin numbers are invariants of oriented topological bordism
in the usual way. In particular, $\hat A(E)=0$, as $E$ is a topological boundary. Because $\hat{A}(S^{n}) = 0$ the 
claim follows. 
\end{proof}

One of the referees pointed out that strongly $\hat{A}$-multiplicative fibers
different from spheres can be obtained using results of Farrell-Jones on rational homotopy types  of 
automorphism groups of hyperbolic manifolds.

\begin{prop} Let $m \geq 10$ and let $F$ be a connected oriented closed hyperbolic manifold of 
dimension $m$.  Then for all $0 \leq k \leq \frac{m-4}{3}$ the manifold $F$ is a
strongly $\hat{A}$-multiplicative fiber
 in degree $k$.  
\end{prop} 

\begin{proof} Let $B$ be a closed smooth  oriented manifold of dimension $k +1$ and $F \to E \to B$ be an oriented 
 smooth fiber bundle classified by a map $B \to B\Diff^+(F)$. We pass to the underlying topological bundle, which is classified
 by the composition $\phi : B \to B \Diff^+(F) \to B \Top^+(F)$.  In the following we abbreviate $\Top^+(F)$ by 
 $\Top$.
 
 By 
 \cite[Corollary 10.16.]{Farrell_Jones} $\pi_0(\Top)$ is a semidirect product 
   of a countable elementary abelian $2$-group and the group ${\rm Out}(\pi_1(F))$. The last group 
   is finite by the Mostow
rigidity theorem. Using the fact that the rational homology of finite and of abelian
 torsion groups vanishes (as each homology class is carried by a finitely
 generated subgroup) 
 we get $\widetilde{H}_*(K(\pi_0(\Top),1);\Q) = 0$. Hence the Leray-Serre spectral sequence for the orientation fibration $\widetilde{B\Top} \to B\Top  \to 
 K(\pi_0(\Top), 1)$ shows that $H_*(\widetilde{B\Top} ; \Q) \cong H_*(B\Top;\Q)$. 
 
  Now the rational Hurewicz theorem can be applied  to the simply connected space 
 $\widetilde{B\Top}$. Because 
 $\pi_s(\widetilde{B \Top}) \otimes \Q = \pi_s(B\Top) \otimes \Q= \pi_{s-1}(\Top) \otimes \Q = 0$ for $2 \leq s \leq \frac{m-4}{3}+1$
by \cite[Corollary 10.16.]{Farrell_Jones}, we hence
obtain $H_s(B\Top;\Q) = H_s(\widetilde{B \Top};\Q) = 0$ for $1 \leq s \leq \frac{m-4}{3} +1$

 With this information we go into the Atiyah-Hirzebruch spectral sequence for oriented bordism with rational coefficients 
 and conclude that rationally 
 the map $\phi$ is oriented bordant to the constant map $B \to B \Top$.
 This implies  that rationally the topological manifold $E$
 is oriented bordant to the product  $B \times F$. Hence 
$\hat{A}(E) = \hat{A}(B \times F) = \hat{A}(B) \cdot \hat{A}(F)$ as claimed. 
 \end{proof} 
 
 Closed hyperbolic spin manifolds do not admit metrics of positive scalar curvature \cite{GLIHES} and so this 
 construction will not be used further in our discussion.

\begin{theorem}\label{theo:main_geom}
  In the situation of part a) of Theorem \ref{main_general} assume in addition that the manifold $M$ is an 
 $\hat A$-multiplicative fiber in degree $k\geq 1$. Then $\pi_{k} ({\Riem}^+(M), g_0)$ contains elements all of whose multiples are geometrically significant.  The images of these elements under the Hurewicz map have infinite order. 
\end{theorem}

\begin{theorem} \label{main_moduli} In the situation of part a) of Theorem \ref{main_general}, 
\begin{itemize} 
  \item[a)]  if in addition $M$
  is a connected $\hat A$-multiplicative fiber in degree $k\geq 0$, then the map 
\[
    \pi_{k} ({\Riem}^+(M), g_0) \to \pi_k( {\Riem}^+(M) / \Diff_{x_0}(M) , [g_0]) 
\]
has infinite image. If $k \geq 1$, the image contains  elements of infinite order. 
  \item[b)] if in addition $M$ is a simply connected strongly $\hat{A}$-multiplicative fiber in degree 
  {$k \geq 0$}, the image of the map 
 \[
    \pi_{k} ({\Riem}^+(M), g_0) \to H_k( {\Riem}^+(M) / \Diff_{x_0}(M)) 
\]
 contains elements of infinite order. 
 \end{itemize} 
   \end{theorem} 

Part b) of Theorem \ref{main_general} is the case $M = S^{4n-k-1}$ of the last result.

The elements in $\pi_k({\Riem}^+(M), g_0)$ constructed in \cite{CS, Hitchin} can in fact be obtained 
by pulling back $g_0$ along families in $\Diff_{x_0}(M)$. Hence these elements are not only not geometrically significant, 
but are even mapped to zero under the map in part a) of Theorem \ref{main_moduli}.

In  \cite{BHSW} 
the authors constructed, for arbitrary $k \geq
1$, non-zero elements in the $4k$-th homotopy groups of the full moduli space
${\Riem}^+(M)/ {\Diff}(M)$ of certain  
closed non-spin manifolds $M$ of odd dimension. These elements can be lifted
to ${\Riem}^+(M)/ \Diff_{x_0}(M)$, but not to ${\Riem}^+(M)$. 

It remains an interesting open problem whether there are examples of manifolds $M$ for which any of the statements 
of Theorem \ref{main_moduli} 
remain valid for  the full moduli space ${\Riem}^+(M) /
\Diff(M)$. 

{Kreck and Stolz \cite{Kreck_Stolz} constructed closed $7$-manifolds $M$ for which 
the moduli spaces ${\rm Sec}^+(M) / \Diff(M)$ of positive sectional curvature metrics 
are not path connected and closed $7$-manifolds for which the moduli spaces ${\rm Ric}^+(M) / \Diff(M)$ of positive Ricci curvature metrics 
have infinitely many path  components. It remains a challenging problem to construct 
non-zero elements in higher homotopy groups of (moduli) spaces of positive sectional and 
positive Ricci curvature metrics.} 

{\em Acknowledgements.} We thank Boris Botvinnik and Peter Landweber for helpful comments on an earlier version of this paper and the anonymous referees for carefully reading our manuscript. 
Their numerous thoughtful remarks  helped us to improve the exposition and 
to eliminate a blunder in the proof of
  Proposition \ref{prop_multiplicative_fiber} (b).

\section{APS index and families of metrics of positive scalar curvature} \label{APS}

If $(W,g_W)$ is a compact Riemannian spin manifold with
  boundary and with 
positive scalar curvature on the boundary (the Riemannian metric 
always assumed to be of product form near the boundary), then the Dirac
operator on 
$W$ with Atiyah-Patodi-Singer
boundary conditions is a Fredholm operator. Equivalently, if one attaches an
infinite half-cylinder $\boundary W\times [0,\infty)$ to the boundary and
extends the metric as a product to
obtain $W_\infty$, the Dirac operator on $L^2$-sections of the spinor bundle
on $W_\infty$ is a Fredholm operator. This uses invertibility of the boundary operator 
due to the positive scalar curvature condition. Both operators have the same
index, the APS index $\ind(D_{g_W})\in \Z$, which depends on the Riemannian metric on the
boundary.  

We collect some well known properties of this index. A detailed discussion can be found in \cite{Melrose}. 

\begin{enumerate} 
  \item If $g_W$ is of positive scalar curvature, we have $\ind(D_{g_W})=
    0$. This follows from the usual
    Weitzenb\"ock-Lichnerowicz-Schr\"odinger  argument.
  \item $\ind(D_{g_W})$ is invariant under deformations of the metric $g_W$ during which the metrics on the boundary maintain positive 
            scalar curvature. This follows from the homotopy invariance of the index of Fredholm operators. 
  \item The index can be computed by the APS index theorem 
  \begin{equation*}
   \ind(D_{g_W}) =\int_{W} \mathcal{\hat A}(W)
  - \frac{1}{2}\eta(\boundary W),
\end{equation*}
which involves the $\hat{A}$-form and the $\eta$-invariant of Atiyah-Patodi-Singer. 
  \item (Gluing formula) If $(V, g_V)$ is another Riemannian spin manifold and there is a spin preserving isometry $\psi \colon\partial V \to - \partial W$, then 
   $\ind(D_{g_V}) + \ind(D_{g_W}) = \hat{A}(V \cup_{\psi} W)$ where $\hat{A}$ is the usual $\hat{A}$-genus of the closed spin manifold $V \cup_{\psi} W$.  
  \end{enumerate}

\begin{defn} Let $M \rightarrow E \to B$ be a smooth fiber bundle of closed manifolds over a compact base manifold
$B$, let $g_B$ be a Riemannian metric on $B$, let 
$\mathcal{H}$ be a horizontal distribution on $E$ and $(g_b)_{b \in B}$ 
be a smooth family of  positive 
scalar curvature metrics on the fibers of $E$. Combining these data we obtain a Riemannian metric $g_E = g_B\oplus (g_b)$ on the total space $E$, where $g_B$ is 
lifted to horizontal subspaces of $E$ using the given distribution $\mathcal{H}$. Note that 
$(E,g_E) \to (B, g_B)$ is a Riemannian submersion. 

Using the O'Neill formulas \cite[Chaper 9.D.]{Besse}  there is an $\epsilon_0 > 0$ with 
the following property: For each $0 < \epsilon \leq \epsilon_0$ the metric $g_B \oplus (\epsilon \cdot g_b)$ 
on $E$ is of positive 
scalar curvature. We call such an $\epsilon_0$ an {\em adiabatic constant} for  the triple $(E, g_B, \mathcal{H})$. 
\end{defn} 

Assume that $B$ and $M$ are closed spin manifolds and $\phi\colon B \to \Riem^+(M)$ is a continuous map. This induces 
a family $(g_b)_{b\in B}$ of positive scalar curvature metrics on $M$, which can be assumed to be smooth after 
a small perturbation. 
After picking some metric $g_B$ on $B$ we can consider 
both the product metric $g_B \oplus g_0$, and
the metric $g_B\oplus ( g_b)_{b \in B}$ on $B\times M$.

We now multiply both fiber metrics on $B \times M \to B$ with adiabatic constants  so that the resulting metrics on $B \times M$ 
are of positive scalar curvature and denote these new metrics by the same symbols. Let 
$g = g_{(B \times [0,1]) \times M}$ be an arbitrary metric on 
$(B \times [0,1]) \times M$ interpolating between the metrics $g_B \oplus g_0$ on $(B \times 0) \times M$ 
and $g_B \oplus (g_b)$ on $(B \times 1) \times M$ (always of product form near the boundaries) and set 
\[
     \hat{A}_{\Omega}(\phi) = \ind(D_g) \in \Z, 
\]
the APS-index of the Dirac operator $D_g$  for the metric $g$ on the spin manifold $(B \times [0,1]) \times  M$. 
Note that $\ind(D_g)$ is equal to the relative index $i(g_B \oplus g_0\, , \, g_B \oplus (g_b))$ of Gromov-Lawson 
\cite[p. 329]{GL}.

The following facts show that the invariant $\hat{A}_{\Omega}$ is well defined on $\Omega^{\Spin}_k(\Riem^+(M))$ and 
hence defines a group homomorphism
\[
   \hat{A}_{\Omega} : \Omega^{\Spin}_k ( \Riem^+(M)) \to \Z.
\]

 \begin{itemize} 
  \item If we choose a different metric on $B$ or scale the metrics $g_0$ and $(g_b)$ by other adiabatic 
           constants, the index  $\ind(D_g)$ remains unchanged by property (2) above. 
  \item If we choose another metric  $g'$ on $(B \times [0,1]) \times M$ interpolating between the two 
           given metrics on the boundary,  the gluing formula implies 
      \[
           \ind(D_g)- \ind(D_{g'}) = \hat{A} ( M \times B \times S^1 ) = 0 
     \]
     using the fact that the $\hat{A}$-genus is multiplicative in products. 
   \item Now assume that $(W,g_W)$ is a compact Riemannian spin manifold of dimension $k+1$ with boundary $B$ and 
             that $(g_w)_{w \in W}$ is 
             a smooth family of positive scalar curvature metrics on $M$  restricting to a
               given family $(g_b)_{b \in B}$ over $B$. We also consider the 
             constant family $(g_0)$ on $M$. With respect to the given metric $g_W$ on $W$ we 
             scale both fiberwise metrics on $W \times M$ by 
             adiabatic constants. 
              
              By the gluing formula 
              and the fact that the APS-indices of both metrics on $W \times M $ vanish, we obtain (with 
             $g = g_{(B \times [0,1])\times M}$ as before)  
       \[
           \ind(D_{g})  = \ind (D_{g_W \oplus g_0}) + \ind(D_g) + \ind(D_{g_W \oplus (g_w)_{w\in W}} ) = \hat{A} (X) ,
      \]
      where
      \[
         X = (W \, \cup_{\partial W = B \times \{0\}} \, B \times [0,1]  \, \cup_{B \times \{1\} = \partial W}  \, W ) \times M.
      \] 
      Because $X$ is spin and $M$ admits a metric of positive scalar curvature we have $\hat{A}(X) = 0$ 
       and hence $\ind(D_g) = 0$, 
       as required. 
\end{itemize}       
      
Let $g_0 \in \Riem^+(M)$. Using composition with the canonical map $\pi_k(\Riem^+(M),g_0) \to \Omega^{Spin}_k(\Riem^+(M))$ we also define 
\[
    \hat{A}_{\pi} \colon \pi_k(\Riem^+(M),g_0) \longrightarrow \Omega^{Spin}_k(\Riem^+(M)) \stackrel{\hat{A}_{\Omega}}{\longrightarrow} \Z 
\]
for all $k \geq 0$. This is a group homomorphism for $k > 0$ and a map of sets for $k = 0$. 

The action of $\Diff(M)$ on $\Riem^+(M)$ by pull-back induces an action of $\pi_k(\Diff(M), \id)$ on $\pi_k(\Riem^+(M), g_0)$. For a pointed map $\phi \colon (S^k, *) \to (\Diff(M), \id)$ and a family $c = (g_t)_{t \in S^k}$ based at $g_0$ the pair $([\phi], [c])$ is mapped to the homotopy class
represented by the family $(\phi(t)^*g_t)_{t \in S^k}$, which we denote by $\phi^*c$ for short. The neutral element $e\in\pi_k(\Riem^+(M), g_0)$ is given by the constant family with value $g_0$.

\begin{prop}\label{prop:Ahat_additivity}
For a closed spin manifold $M$ and $[\phi]\in \pi_k(\Diff^+(M), \id)$ where $k\geq 1$ we have
\[
   \hat A_{\pi} (\phi^* e) = \hat A_{\rm Diff} (\phi)
\]
with $\hat A_{\rm Diff}(\phi)$ as defined before Corollary \ref{corol:non-mult-Ahat}. If $k=0$, the same equation holds provided $[\phi]$ is represented by a spin-preserving diffeomorphism.
\end{prop}

\begin{proof}
In the case that  the map $S^k \times M \to S^k \times M$ adjoint to $\phi$ is spin preserving, the assertion follows from the gluing formula for the APS index. This shows in particular the last claim of the proposition. 

Moreover we observe that if $k\geq 1$, both sides of the asserted equation define  group 
homomorphisms $\pi_k(\Diff^+(M) , \id) \to \Q$ --- for the left hand side we use that it is given by the composite of group homomorphisms
\[\pi_k(\Diff^+(M),\id) \to \pi_k(\Riem^+(M), g_0) \xrightarrow{\hat A_\pi} \Z\subset \Q\]
where the first map is induced by the action of the diffeomorphism group on the base-point $g_0\in\Riem^+(M)$. The proof is concluded by the fact that each diffeomorphism $S^k \times M \to S^k \times M$ has a power which is spin preserving.
\end{proof} 

{We obtain the following corollaries for 
a closed spin manifold $M$ which is an $\hat{A}$-multiplicative fiber in degree $k$.}

\begin{cor}\label{corol:geometric_significance} If, {for $k \geq 1$},  an element $c\in\pi_k(\Riem^+(M), g_0)$ is not geometrically significant then $\hat A_{\pi}(c)=0$.
\end{cor}

By considering the long exact sequence in homotopy associated to the fibration 
\[
     \Riem^+(M)\hookrightarrow  \Riem^+(M) / \Diff_{x_0} (M) \to B\Diff_{x_0}(M) 
\]
we also have

\begin{cor}\label{corol:observer_non_triv} {Let $M$ be a connected. Then for 
$k \geq 1$}  the map $\hat A_{\pi} \colon \pi_k (\Riem^+(M), g_0)\to \mathbb{Z}$ factors through the image of the projection-induced map
\[\pi_k (\Riem^+(M), g_0) \to \pi_k (\Riem^+(M)/\Diff_{x_0}(M), [g_0]).\]
\end{cor}

In particular, the group $\pi_k (\Riem^+(M)/\Diff_{x_0}(M), [g_0])$ contains 
an element of infinite order if the map $\hat A_{\pi}$ is non-zero in degree $k\geq 1$. As the isotopy classes of spin-preserving diffeomorphisms form a finite-index subgroup of $\pi_0 ( \Diff(M)) $, we also deduce:

\begin{corollary}\label{cor:case_pi_0} {For $k = 0$} 
the  set $\pi_0 (\Riem^+(M)/\Diff(M))$ is infinite if the image of $\hat
A_{\pi}$ in degree $k=0$ is infinite.
\end{corollary}

In passing we also note the following consequence. 

\begin{cor} \label{insignificant} Let $F$ be a manifold as in Remark \ref{alpha}  admitting a positive scalar curvature metric $g_0$ 
and appearing as fiber in $F \to P \to S^{k+1}$. Then $\pi_k(\Riem^+(F), g_0)$ contains elements of infinite order 
(and infinitely many elements for $k = 0$), 
which are not geometrically significant. 
\end{cor}

\section{Fiberwise Morse theory and metrics of positive scalar curvature}

Let $B$ be a closed   smooth connected manifold and let $W$ be a smooth
$n+1$-dimensional compact manifold with two
connected boundary components $M_0 = \partial_0 W$ and $M_1 = \partial_1
W$. We assume throughout this section that $\dim W = n+1 \geq 6$.
Let 
\[
    W \hookrightarrow E \stackrel{\pi}{\to} B 
\]
be a smooth fiber bundle with structure group ${\Diff} (W; M_0, M_1)$
consisting of diffeomorphisms $W \to W$ mapping $M_i$ to $M_i$ for $i=0,1$ and
preserving pointwise some collar neighborhoods of $M_0$ and $M_1$ in $W$. 
Then the total space $E$ has again two boundary components $\partial_i E$ which are total spaces of fiber bundles 
$M_i  \hookrightarrow \partial_i E \to B$. 

In the following discussion, the notion {\em fiberwise} in connection with a mathematical object defined on $E$ 
is a shorthand for the fact that this object is defined or constructed on each fiber $E_t = \pi^{-1}(t) \subset E$, $t \in B$, 
 but still defines a global object on $E$. In this terminology, 
the {\em  fiberwise tangent bundle of $E$} is equal to the {\em vertical tangent bundle} 
\[
    T^{vert} E = \ker (T\pi \colon TE \to TB) \to E 
\]
and a {\em fiberwise Morse function} is a smooth function $E \to \R$ which restricts to a Morse  function on each fiber $E_t$. 

Theorem \ref{Walsh_Existence} below states when a fiberwise metric of positive scalar curvature
on $\partial_0 E$, i.e.~a smooth metric on $T^{vert} (\partial_0 E)$ that restricts to a metric of positive scalar curvature on each fiber $(\partial_0 E)_t$,   can be extended
to a fiberwise metric of positive scalar curvature on $E$. This  is 
a fibered version of the well known fact, due to Gromov-Lawson, Schoen-Yau and Gajer,  that if $M_1 \hookrightarrow W$ is a $2$-equivalence (inducing a  bijection on $\pi_0$ and $\pi_1$ 
and a surjection on $\pi_2$), then a positive scalar curvature metric on $M_0$ can be
extended to $W$, cf. \cite[Theorem 3.3]{StolzICM}.  This uses a handle decomposition of $W$ induced by a Morse function. 
 In a fibered situation the situation is more complicated, because  it is in general not possible 
to construct a fiberwise Morse function on $E \to B$.  We handle this situation by combining Igusa's theory of 
fiberwise generalized Morse functions \cite{Igusa2} with Walsh's generalization of the Gromov-Lawson surgery 
method to a fibered situation \cite{Walsh}. 

The following discussion summarizes  \cite[\S 3]{Igusa2} in a form needed for our purpose. Let $M$ be a compact smooth manifold. Recall that a smooth function 
\[
    f \colon M \to \R 
\]
is a {\em generalized Morse function} if the gradient of $f$ is transverse to
$\partial M$ and if
each critical point $p \in M$ of $f$ is either of Morse or birth-death type. By definition, in the first case
there are local coordinates $(x_1 , \dots, x_i, y_1, \dots, y_j)$ around $p$ so that $p$ has coordinates $(x(p), y(p)) = (0,0)$ and $f$ can 
locally be written as
\[
     f(x,y)  = f(p) - x_1^2 - \dots - x_i^2 + y_1^2 + \dots + y_j^2 \, . 
\]
In the second case there are local coordinates $(x, y_1, \dots, y_k, z_1, \dots, z_l)$ around $p$ so that $p$ has coordinates $(x(p), y(p), 
z(p)) = (0,0,0)$ and  $f$ can locally be written in the 
form 
\[
    f(x,y,z)  = f(p)  + x^3 - y_1^2 - \dots - y_k^2 + z_1^2 + \dots + z_l^2 \, . 
\]
The type of the critical point as well as the numbers $i$ and $k$ are uniquely determined by $p$ and $f$. In the first case (of a Morse singularity)  
we say that $p$ is a critical point of {\it index} $i$, in the second case (of a birth-death singularity) we say that $p$ is a critical point of 
{\em index} $k +1/2$. This is motivated by the fact that  in a parametrized
family of Morse functions birth-death singularities typically arise in the situation when
two critical points of index $k$ and $k+1$ cancel each other along a
one-parameter sub-family. The normal form for such a family, parametrized by $t \in (-\epsilon, \epsilon)$,  is
$f_t(x,y,z)=f_0(p) +x^3-tx - y^2+z^2$ with a birth-death singularity at $t = 0$. Note that in
\cite{Igusa2} the index of a birth-death singularity as above  
is defined to be $k$.

\begin{defn}
By a {\em generic family of generalized Morse functions} on the bundle $E \to B$ we mean a smooth function 
\[
   F  \colon E \to [0,1]
\]
with the following properties: 
\begin{itemize} 
  \item $\partial_0 E = F^{-1}(0)$, $\partial_1 E = F^{-1}(1)$.
  \item The function $F$ is fiberwise generalized Morse, i.e. for each $t \in B$ the restriction 
  \[
       f_t = F |_{E_t} \colon E_t \to [0,1]
  \]
   is a generalized Morse function. 
  \item The birth-death points in each $f_t$ are {\em generically unfolded},
    i.\,e. they represent a transversal intersection 
    of the fiberwise $3$-jet defined by $F$ and the subspace of fiberwise
    birth-death $3$-jets inside the bundle of 
 fiberwise $3$-jets of smooth functions $E \to \R$, compare \cite[\S 2]{Igusa2}.
 \end{itemize} 
\end{defn}
This is essentially a global version of \cite[Definition 3.1]{Igusa2} on a
non-trivial (as opposed to a product) bundle $E$. However, we do not require
``linear independence of the birth directions of the different birth-death
points of a fixed fiber $E_t$''. Instead we follow
\cite[Definition 4.8]{Walsh} at this point.

For a smooth function $F \colon E \to [0,1]$ we denote the subset of fiberwise
critical points in $E$ by $\Sigma(F)$. We remark that if
$F$ is a generic family of generalized Morse functions then the subsets
 \[
      A_1(F)  \subset E, \quad A_2(F) \subset E
 \]
 of fiberwise Morse and birth-death singularities are smooth submanifolds of
 $E$ of dimension $\dim B$ and $\dim B - 1$, respectively.
The manifold $A_2(F)$ is  closed and $\Sigma(F) = A_1(F) \cup A_2(F)$. 
Furthermore,  $\pi$ restricts to an immersion $\pi|\colon A_2(F) \to B$. For these
statements, compare \cite[Lemma 3.3]{Igusa2}. The additional property of
self-transversality of $\pi|\colon A_2(F)\to B$ obtained in \cite[Lemma
3.3]{Igusa2} cannot be assumed in our situation because we do not
insist on linear independence of different birth-directions in a fixed fiber. Also we 
remark that in \cite{Walsh} the sets $A_1(F)$ and $A_2(F)$  
are denoted $\Sigma^0$ and $\Sigma^1$, respectively.

 We state the following fundamental existence result. 
 
 \begin{prop} \label{existence} Let $\dim W > \dim B$. Then the generic families of generalized Morse functions $F\colon E \to [0,1]$ form
 a nonempty open subset of $C^{\infty}(E, [0,1])$ equipped with the $C^{\infty}$-topology. 
 \end{prop}

 \begin{proof} The main result of \cite{Igusa1} implies that the space of smooth functions $E \to [0,1]$ which restrict to 
 generalized Morse functions on each fiber is nonempty in $C^{\infty}(E, [0,1])$. The assertion now follows from 
 multijet transversality in the same way the corresponding \cite[Lemma
 3.2]{Igusa2} follows. 
 \end{proof} 
 
 The next proposition shows that we have good control of the index of critical points (of Morse or birth-death type) of $f_t$, $t \in B$, for a generic family of generalized Morse 
 functions $F\colon E \to [0,1]$. Following \cite[Section 4]{Walsh} we call
 such a family {\em admissible} if all critical points are of index $\leq n-2$
 for each $t$. (Recall that $\dim W = n+1$, which is in accordance with the
 convention in \cite{Igusa2}).
 
 \begin{prop} \label{cancellation} Assume that $\dim W  \geq  2 \dim B + 5$ and
 that the inclusion $M_1 \hookrightarrow W$ is $2$-connected. Then the space
 of admissible generic families of generalized Morse functions $E \to [0,1]$
 is a nonempty open subset of $C^{\infty}(E, [0,1])$. 
  \end{prop}
 
 \begin{proof} This follows from the proof of Proposition \ref{existence} by applying a variant of Hatcher's two-index theorem,
see \cite[Corollary VI.1.4]{Igusa2} (with $i=0$, $j = n-1$ and $k = \dim B$ in the notation of {loc.~cit.}) 
before appealing to multijet transversality. Here we note that the subset of
generalized Morse
functions on $W$ whose Morse critical points have indices at most $n-2$ and
birth-death critical points have indices at most $n-2\frac{1}{2}$ (resp. $n-3$ in the
notation of \cite{Igusa2}) is open in the set of all $C^\infty$-maps.  
\end{proof}

We note the following addendum which is proved by the same methods.

\begin{add} \label{addendum} If in the situation of Proposition \ref{cancellation} the inclusion $M_0 \hookrightarrow 
W$ is also $2$-connected, then the space of generic families of generalized Morse functions $f \colon E \to [0,1]$ with 
the property that both $f$ and $1-f$ are admissible is a nonempty open subset of $C^{\infty}(E, [0,1])$. 
\end{add}

For our later discussion we need convenient coordinates around the fiberwise singular set  $\Sigma(F)$ for generic families of generalized Morse functions $F$, 
cf. \cite[Definition 4.3]{Walsh}. Because the normal bundles of 
$A_1(F) \subset E$ and $A_2(F) \subset E$ are in general nontrivial, we need
to work in a twisted setting.

Choose a fiberwise Riemannian  metric $g^{vert}$ on $E$, i.\,e. a smooth section of $(T^{vert} E)^* \otimes (T^{vert} E)^*$ that defines 
a Riemannian  metric on each  $E_t$.  At a critical point $p \in E_t$ of $f_t$ we obtain 
an orthogonal (with respect to $g^{vert}$) splitting  $V^+_p \oplus V^-_p \oplus V^0_p$   into positive, negative and null eigenspaces of the (fiberwise) Hessian of $f_t \colon E_t \to \R$ on $T^{vert}_p E$. 
The space $V^0_p$ is zero if $p \in A_1(F)$ and one-dimensional if $p \in
A_2(F)$. Hence we obtain vector bundles
\[
   V^{\pm}  \to A_1(F) , \quad V^{\pm} \to A_2(F) , \quad V^0 \to A_2(F)  
\]
so that $V^+ \oplus V^-$ has structure group $O(i) \times O(j)$ on path components of $A_1(F)$ and 
$V^+ \oplus V^-\oplus V^0$ has structure group $O(k) \times O(l) \times SO(1)$ on 
path components of $A_2(F)$.  The numbers $i$ and $k$ refer to indices of critical points as before and $V^0$ is a 
one-dimensional real bundle.  These
indices can of course vary on different path components of  $A_1(F)$ and $A_2(F)$. Note that $V^0 \to A_2(F)$ carries a preferred orientation pointing  in the direction of increasing $f_t$ for $t \in A_2(F)$. 

In order to describe the behavior of critical points around a birth-death singularity $p \in E_t$  in a coordinate-independent way 
it is useful to choose an additional  Riemannian  metric $g_B$ on $B$ and a
horizontal distribution on $TE$ together with 
the lift $g^{hor}$ of $g_B$ to this horizontal distribution. Together with the metric $g^{vert}$ we hence obtain a Riemannian  submersion 
\[
     (E, g^{vert} \oplus g^{hor}) \to (B, g_B) \, . 
\]
Fixing this, for small enough $\epsilon > 0$ we obtain  canonical embeddings 
\[
   \xi \colon A_2(F) \times (-\epsilon, \epsilon)   \to E 
\]
restricting to the inclusion $A_2(F) \hookrightarrow E$ on $A_2(F) \times \{0\}$ and so that for each $p \in A_2(F)$ the restricted map 
$\xi \colon \{ p \} \times (-\epsilon, \epsilon)   \to E$ describes the unique horizontal curve of unit speed which maps to a geodesic in $B$ orthogonal 
to $D_p\pi(T_p  A_2(F)) \subset T_{\pi(p)} B$ and points into the birth direction of the birth-death singularity $p$.

For such an $\epsilon > 0$ consider the vector bundle of rank $n+1$ 
\[
      V^{\epsilon} =  {\rm pr}_1^* (V^0 \oplus V^- \oplus V^+) \to A_2(F) \times (-\epsilon, \epsilon) 
\]
i.e. the pull-back of $T^{vert}E|_{A_2(F)}$ to $A_2(F) \times (-\epsilon, \epsilon)$ along the projection 
${\rm pr}_1 \colon A_2(F) \times (-\epsilon, \epsilon) \to A_2(F)$. The vector bundle $V^{\epsilon}$ 
splits in a canonical way as 
\[
    V^{\epsilon} = V^{\epsilon,0} \oplus V^{\epsilon,-} \oplus V^{\epsilon,+} \, . 
\]
We equip $V^{\epsilon}$ with the pull-back metric of the induced metric on $T^{vert}E |_{A_2(F)}$ and consider for $\delta > 0$ the 
disc bundle $D_{\delta}(V^{\epsilon}) \to A_2(F) \times (-\epsilon, \epsilon)$. An {\em extended tubular neighborhood} around 
$A_2(F)$ is a smooth embedding 
\[
    \phi \colon D_{\delta}(V^{\epsilon}) \hookrightarrow E 
\]
so that the diagram 
\[
   \xymatrix{
    D_{\delta}(V^{\epsilon})  \ar[d]    \ar[r]^-{\phi} & E \ar[d]^{\pi} \\
    A_2(f) \times (-\epsilon, \epsilon) \ar[r]^-{\pi \circ \xi} & B  }
\]
commutes. This is exactly the global version (on non-product bundles) of the
normal form defined in \cite[Appendix, Lemma 3.5]{Igusa2}. 

After choosing some local isometric trivialization 
\[
    E^{vert}|_U \cong U \times \R \times \R^k \times \R^l
\]
over  an open subset $U \subset A_2(F)$, where 
the isomorphism on the $\R$-factor is orientation preserving, we obtain a corresponding local 
isometric trivialization 
\[
     V^{\epsilon}|_{U \times (-\epsilon, \epsilon)} \cong(  U \times (-\epsilon, \epsilon) ) \times \R \times \R^{k} \times \R^l 
\]
so that we can work with local coordinates $(p,s,x,y,z)  \in A_2(F) \times (-\epsilon, \epsilon) \times \R \times \R^k \times \R^l$. For the following definition compare \cite[Appendix, Definition 3.9]{Igusa2}. 

\begin{defn} \label{normal} Let $g^{vert}$ be a fiberwise Riemannian  metric on
  $E$. We call a generic family of generalized Morse
functions $F \colon E \to [0,1]$ in {\em normal form} with respect to
$g^{vert}$, if there are numbers $\sigma, \tau > 0$,
$\delta > 3\sigma$ and $\epsilon > \sigma^2$ 
and an extended tubular neighborhood 
\[
    \phi \colon D_{\delta}(V^{\epsilon})  \hookrightarrow E 
\]
(with respect to some horizontal metric $g^{hor}$ as before) so that 
\begin{itemize} 
   \item $\phi$ is fiberwise isometric,  
   \item for all $\{ (p, s, x,y,z) \in D_{3 \sigma}(V^{\epsilon}) ~| ~
|s| \leq \sigma^2\}$ the function $F \circ \phi$ is given by 
\[
  (s, x,y,z) \mapsto F(p ) + g_s(x) - \| y\|^2 + \| z \|^2 
\]
where 
\[
    g_s \colon \R \to \R \, , ~ s \in [-\sigma^2, \sigma^2]
\]
is the smooth family of functions of \cite[Appendix, Lemma 3.7]{Igusa2} (with the chosen $\sigma$ and $\tau$),  
\item for each critical point
\[
   q \in A_1(F) \setminus \phi \big( \{ (p, s, x,y,z) \in D_{3 \sigma}(V^{\epsilon}) ~| ~ |s| \leq \frac{2}{3} \sigma^2\} \big) 
\]
there is a $g^{vert}_{\pi(q)}$ isometric embedding $\mu
\colon D_{\tau}^{n+1} \to E_{\pi(p)}$ mapping
the midpoint of the disc $D_{\tau}^{n+1} \subset \R^{n+1}$ of radius $\tau$ (which is equipped with the standard metric) to $q$ and so that $F$ is given by
\[
    F(\mu(z_1, \ldots, z_{n+1})) = F(q) + \sum_{i=1}^{n+1} \pm z_i^2 \, , 
\]
where $(z_1, \ldots, z_{n+1})$ are the standard coordinates of $\R^{n+1}$. Furthermore, the embedding $\mu$ can be assumed 
to vary smoothly on a contractible neighborhood of $q$  in $A_1(F)$.  
\end{itemize} 
\end{defn} 

We obtain the following existence result. 

\begin{prop} \label{generic_existence} Assume
that $\dim W  \geq  2 \dim B + 5$ and that the inclusion $M_1 \hookrightarrow W$ is $2$-connected.  Then there is a fiberwise
Riemannian  metric $g^{vert}$ on $E$ and a generic family of generalized Morse
functions $F \colon E \to [0,1]$ which is admissible and in normal form with 
respect to $g^{vert}$. 
\end{prop} 

\begin{proof} The proof is along the lines of the proof of \cite[Appendix, Theorem 3.10]{Igusa2}. Here we notice that in Igusa's work this 
proof is given for the case when the bundle $E = W \times B \to B$ is trivial, but under the general assumption 
that $E^{vert}|_{A_1(F)}$ and $E^{vert}|_{A_2(F)}$ are non-trivial. Because the deformations 
of $f$ and $g^{vert}$ are carried out {loc.~cit.} in a coordinate-independent way, 
the same proof can be used to treat the general case of the nontrivial bundle $W \hookrightarrow E \to B$ appearing in 
our discussion. 
\end{proof}

Combining this result with  \cite[Theorem 1.4]{Walsh}  we obtain

\begin{thm} \label{Walsh_Existence} Assume
that $\dim W  \geq  2 \dim B + 5$ and that the inclusion $M_1 \hookrightarrow W$ is $2$-connected. Also assume 
that there exists a fiberwise metric of positive scalar curvature on a fiberwise
collar neighborhood of $\partial_0 E \subset E$ which is fiberwise of product
form on this collar  
neighborhood. Then this 
metric can be extended to a fiberwise metric of positive scalar curvature on
$E \to B$ which is fiberwise of product form 
on a fiberwise collar neighborhood of $\partial_1 E$ in $W$ . 
\end{thm}

\section{Homotopy classes of positive scalar curvature metrics} 
\label{proof_main}

Using the invariant $\hat{A}_{\Omega}$ from Section \ref{APS}  it is now rather straightforward to prove Theorems
\ref{main_general}, 
\ref{theo:main_geom} and \ref{main_moduli}, assuming Theorem
\ref{technical}.

We fix $k \geq 0$ and
pick a bundle $P \to S^{k}$ whose total space is of non-zero $\hat{A}$-genus and 
of dimension $4n$, where we also fix $l=2$ and
require $4n\ge 3k+5$ (i.\,e. the fiber dimension satisfies $4n - k \geq 2k + 5$). Note that this and Theorem \ref{technical} determines
the lower bound $N(k)$ on $n$ for our main theorems. We assume 
conditions (\ref{item:high_connect}) and (\ref{item:section}) of Theorem
\ref{technical}. 

Let $s \colon S^{k} \to P$ be a smooth section with trivialized  normal 
bundle. Inside the resulting embedding of $S^k\times D^{4n-k}$ we construct
an  embedding
\[
   \rho \colon  S^{k} \times \left( D^{4n-k}  \cup_{D^{4n-k-1}\times 0 }   D^{4n-k-1}\times
   [0,1]  \cup_{D^{4n-k-1} \times 1} D^{4n-k}\right)  \hookrightarrow P .
\]
 Removing the interiors of the 
two copies  of $S^k\times D^{4n-k}$  yields a fibration 
\[
    W \hookrightarrow E \to  S^{k} 
\]
where $E$ has two boundary components    $\partial_0 E$ and $\partial _1 E$
each of which 
may be identified (by the map $\rho$) with the total space of the trivial
fibration  $S^{k} \times  S^{4n-k-1} \to S^{k}$. Furthermore the bundle $E$  
 comes with a fiberwise embedding of $S^k\times D^{4n-k-1}\times
[0,1]$ meeting $\partial_0 E$ and $\partial_1 E$ in 
$S^k\times D^{4n-k-1}\times 0$ and $S^k\times D^{4n-k-1}\times 1$,
respectively. A typical fiber $W$ is displayed in Figure 2.

We use this embedding to form the fiberwise
connected sum of $E \to S^k$ with the trivial bundle $S^k \times (M\times [0,1])\to S^k$ (the interval $[0,1]$ being embedded 
vertically in each fiber of $E$) to obtain a new fiber 
bundle 
\[
    W_M \hookrightarrow  E_M \to S^k
\]
where each fiber $W_M$ is a bordism from $M$ to $M$. Figure 2  illustrates how a
typical fiber of this bundle emerges from the connected sum of $W$ and $M \times [0,1]$. 

\begin{figure}\label{fig:single_fiber}
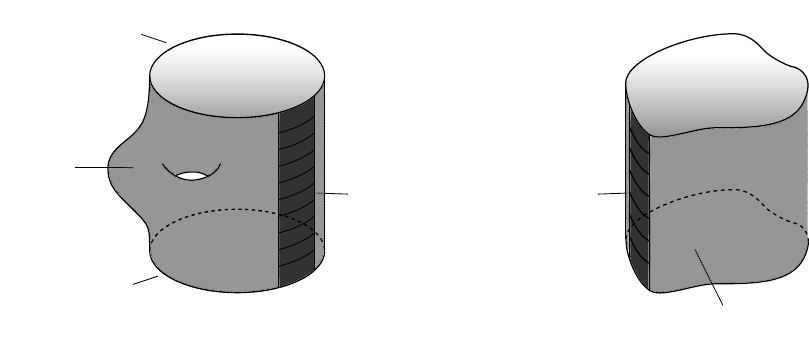
\caption{}
\end{figure}

To keep our notation short we drop the index $M$ from now on 
and call this new bundle $W \to E \to S^k$ again.

If $l = 2$ and the fiber dimension satisfies $4n-k \geq  2k+5$, we can apply
Theorem \ref{Walsh_Existence} so as to extend the constant 
fiberwise positive 
scalar curvature metric $g_0$ on $\partial_0 E$ to a fiberwise 
positive scalar curvature metric on $E$ which is fiberwise of product form
near $\partial_1 E$. 

In view of the fact that $\partial_1 W$ is identified with the trivial bundle
$S^{k} \times M \to S^{k}$ we obtain a new family 
\[
    \phi \colon S^{k} \to {\Riem}^+(M) 
\]
of positive scalar curvature metrics on $M$. 

Unfortunately this need not be in the path component of $g_0$ so that we modify our 
construction as follows.

Let $F \colon E \to [0,1]$ be the generic family
of generalized Morse functions in normal form that was used for the
construction of the fiberwise  
metric of positive scalar curvature on $E$. Then  the image of the set of
birth-death singularities 
\[ 
  \pi(A_2(F)) \subset S^{k} 
\]
is an immersed submanifold of dimension $k-1$ and hence not equal to
$S^{k}$. Let $t \in S^{k}$ be a point not lying in this image. The
singularities of the  restriction $f_t \colon E_t \to [0,1]$ are only of Morse
type.  
In view of Addendum \ref{addendum} we can assume that they have not only coindex at least $3$ but also index at
least $3$. 
We set $W=E_t$, consider the constant family of Morse
functions
\begin{equation*}
   F^{const} \colon  S^{k} \times W  \to     [1,2]   \, ,  \qquad
           (c,w)  \mapsto  2- f_t(w) 
\end{equation*}
and the resulting family of Morse functions
\[
   F \cup  F^{const}   \colon E \cup_{\partial_1 E = S^{k} \times \partial_1 W } ( S^{k} \times W ) \to [0,2]  
\]
on the bundle $E$ together with an upside down copy of the trivial bundle
$S^{k} \times W\to S^{k}$.  
Using \cite[Theorem 1.4]{Walsh}, we can extend the family of positive
scalar curvature metrics on $E$ to the new fiber bundle 
\[
   E' = E \cup_{\partial_1 E = S^{k} \times \partial_1 W} ( S^{k}
  \times W)  \to S^{k} \, . 
\]

Note that the restriction of the Morse function $F \cup F^{const}$ to the fiber $E'_t$ over $t$ is the smooth function 
\[
   W \cup_{\partial_1 W = \partial_1 W} W \to [0,2]
\]
given by $f_t$ on the first and by $2 - f_t$ on the second copy of $W$. It therefore induces a handle 
decomposition in which each handle of index $i$ in the first copy of $W$ corresponds to a handle 
of coindex $i$ in the second copy (in a canonical way). By
  the results of Walsh  the positive
scalar curvature metric
obtained on $E'_t$ by the method from \cite{Walsh}, which on the single fiber $E'_t$ restricts to the classical construction 
from \cite{GL}, can be assumed to coincide on the two copies of $W$ (glued together at $\partial_1 W$).
In particular we can assume that the resulting metric on $E'_t$  restricts to $g_0$ on both ends.

This means that the family of positive scalar curvature metrics 
\[
   \phi' \colon S^k \to {\Riem}^+(M) 
\]
on the upper boundary component $\partial_1 E'$ is in the path component of $g_0$.

\begin{prop}
  For this family of  metrics $\phi'$ we get 
  \[
       \hat A_\pi(\phi') = \hat A(P)\neq 0, 
  \]
 where $P$ is the total space of the bundle from Theorem \ref{technical}. 
\end{prop}

\begin{proof}
  By definition 
\[
      \hat{A}_{\pi}(\phi') = \ind(D_g) , 
\]
where $g$ is a metric on  $(S^k \times [0,1]) \times M$ 
interpolating between the constant family $(g_0)$ on $(S^k \times 0) \times M$ and the family $\phi$ on 
$(S^k \times 1) \times M$, 
both families being scaled by an adiabatic constant with respect to some metric on $S^k$.

  We can isometrically glue $(S^k \times [0,1]) \times M$ along 
  top and bottom to $E'$ to obtain a new spin manifold $P'$. Because $E'$ carries a metric of positive scalar curvature, the gluing formula for the APS index yields
  \[
      \ind(D_{g}) = \hat{A}(P'). 
  \]

Notice that the underlying smooth manifold $P'$ is just obtained from $E'$ by identifying the two boundary components. A fiber $P'_x$ over $x\in S^k$  is depicted in Figure 3.

\begin{figure}\label{fig:single_fiber_2}
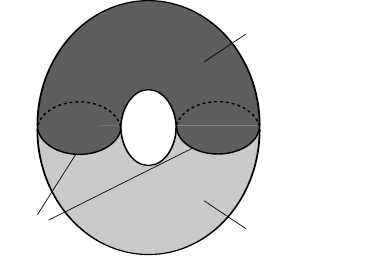
\caption{}
\end{figure}

Let us now write $P'=P'_M$ to denote the dependency on $M$. By construction $P'_M$ is obtained as the  fibered connected sum of $P'_{S^{4n-k-1}}$ and the trivial bundle $S^k\times M \times S^1$, which has trivial $\hat A$-genus. By bordism invariance we conclude that $\hat A(P'_M)$ is independent of $M$, so we may assume that $M=S^{4n-k-1}$, i.\,e. that no connected sum construction has been performed.

In this case we carry out fiberwise two coindex-zero surgeries on the two  copies of $M=S^{4n-k-1}$ inside the fibers of $P'_{S^{4n-k-1}}$ as 
shown in Figure 3. We get a bordism from $P'_{S^{4n-k-1}}$ to the disjoint union of the manifolds $P$ and $S^k\times F$, where $F$ is the fiber of the bundle $F \to P \to S^k$ we started with. The claim follows since the manifold $S^k\times F$ has again vanishing $\hat A$-genus.
\end{proof}

Because $\hat{A}_{\pi}$ is a homomorphism 
we can conclude that $\pi_k(\Riem^+(M), g_0)$ contains elements of infinite order if $k > 0$. 
For $k = 0$ we use the fact that $\hat{A}(P)$ can assume infinitely many integer values for 
different bundles $F \to P \to S^0$ so that $\pi_0 ( \Riem^+(M)) $ is infinite. 

Now consider the chain of group homomorphisms (for $k=0$ the first map is just a map of sets) 
\[
    \pi_k(\Riem^+(M), g_0) \to \Omega^{f\!r}_k(\Riem^+(M))\to \Omega^{Spin}_k(\Riem^+(M)) \to H_k(\Riem^+(M))
 \]
where the superscript ``fr'' stands for framed bordism. Because 
the invariant $\hat A$ is integer-valued and defined on Spin bordism, 
it follows that the images of our non-trivial homotopy classes remain non-zero
in rational framed bordism. The composed map between rational framed bordism and rational 
singular homology being an isomorphism we conclude the proof of part a) of
Theorem \ref{main_general}. 

Together with Corollary \ref{corol:geometric_significance} we also obtain Theorem \ref{theo:main_geom}.  Part a) of Theorem \ref{main_moduli} follows with Corollaries \ref{corol:observer_non_triv} and \ref{cor:case_pi_0} and part b) of Theorem \ref{main_moduli} is a consequence of the following result.

\begin{prop}\label{corol:homology_modu}
  Let $M$ be a simply connected spin manifold which is a \emph{strongly} $\hat A$-multiplicative
  fiber in degree $k$ and admits a metric of positive scalar curvature $g_0$. 
  Then 
  the invariant  $\hat A_{\Omega}$ factors through the image of the canonical map 
\[
   \Omega^{\rm fr}_k(\Riem^+(M))\to \Omega^{\rm fr}_k(\Riem^+(M)/\Diff_{x_0}(M)).
\]
\end{prop}

\begin{proof}
  Precomposing the canonical map from framed to spin bordism 
  yields a map 
 \[
     \hat{A}_{\rm fr} \colon \Omega^{\rm fr}_k(\Riem^+(M)) \to \Z . 
 \]
 We note that each diffeomorphism in $\Diff_{x_0}(M)$ canonically lifts to
 a spin diffeomorphism, because $M$ is simply connected
  by assumption and the differential is the identity at $x_0$.
 
 Let $B$ be a closed stably parallelizable manifold of dimension $k$, let $\phi \colon B \to \Riem^+(M)$ be a 
 continuous map and assume that there is a compact stably parallelizable manifold $Y$ with boundary $B$ 
 together with a map $\Phi  \colon Y \to \Riem^+(M) / \Diff_{x_0}(M)$ so that 
 \[
      \Phi|_{\partial Y } = \pi \circ \phi 
\]
where $\pi \colon \Riem^+(M) \to \Riem^+(M) / \Diff_{x_0}(M)$ is the canonical projection. 
We need to show that $\hat{A}_{\rm fr} ( \phi) = 0$.

  On the one hand, because
  $\Riem^+(M)\to \Riem^+(M)/\Diff_{x_0}(M)$ is a fiber bundle projection, the map 
  $\Phi$ gives rise to a (non-trivial) smooth bundle $E\to Y$ with fiber $M$
  and structure group $\Diff_{x_0}(M)$ equipped with a fiberwise metric of positive 
  scalar curvature. As the map $\Phi$ 
  restricts to $\pi \circ \phi$ on the boundary, the bundle $E|_{\boundary Y}$ admits a trivialization
  $E|_{\boundary Y}\iso S^k\times M$ such that the family of metrics coincides
  with the one given by $\phi$.

  On the other hand we consider the trivial bundle $Y \times M \to Y$ equipped with the constant 
  fiberwise metric $g_0$ of positive scalar curvature. 
  
  After choosing a Riemannian  metric on $Y$, a 
  horizontal distribution on the bundle $E \to Y$ and scaling the fiberwise metrics 
  on $E \to Y$ and ${Y \times M} \to Y$ by an adiabatic constant, we get positive scalar curvature metrics 
  $g_E$ and $g_{Y \times M}$  on the total spaces $E$ and $Y \times M$. Both of these total spaces
  admit canonical spin structures, so that the APS indices of $E$ and $Y \times M$ equipped with these metrics vanish.

  Choose a metric $g$ on $(B \times [0,1]) \times M$ inducing the restriction of  $g_{Y \times M}$ 
  on $(B \times 0) \times M$ and the restriction of $g_E$ on 
  $(B \times 1) \times M$.
  
  From this we obtain a fiber bundle with fiber $M$,  total space 
  \[
    X =   Y \times M \bigcup_{\partial Y \times M = B \times 0 \times M}  B \times [0,1] \times M 
       \bigcup_{B \times 1 \times M = \partial E} E 
  \]
 and base
 \[
       Y \bigcup_{\partial Y = B \times 0 } B \times [0,1] \bigcup_{B \times 1 = \partial Y} Y .
 \]
  We have $\hat{A}(X) = 0$, because $M$ is assumed to be a strongly 
  $\hat{A}$-multiplicative fiber in degree $k$ and the base manifold of this
  bundle is parallelizable.
  Hence
  \[
      \hat{A}_{\rm fr}(\phi) = \ind(D_{g_{Y\times M}}) + \ind(D_g) + \ind(D_{g_E}) = \hat{A}(X) = 0 \, . 
  \]
  as required. 
 \end{proof}

\section{Proof of Theorem \ref{technical}} 

Let us first assume $k \geq 1$, the case $k = 0$ being postponed to the end of the proof. 
We construct $P$ is in several steps. Let $\alpha  \geq
k/4+2$ be a natural number, let either $\beta = \alpha +1$ or $\beta =\alpha
+2$ and
$n = \alpha + \beta$. We consider the trivial fibration
\[
    \phi_0 \colon P_0 = S^{k} \times S^{4\alpha-k} \times S^{4
        \beta}  \to S^{k}   
\]
with (path connected) total space of dimension $4n$. By choosing $\alpha$ appropriately we can assume in addition that the fiber 
$S^{4\alpha - k} \times S^{4 \beta}$ of $\phi_0$ is $l$-connected. 

We first apply the following result of surgery theory, in which $\tau\colon P_0 \to B\pi_1(P_0)$ denotes the classifying map of the universal covering.

\begin{thm}[{\cite[Theorem 6.5]{Davis}}]\label{Davis_theorem}
Let $\mathcal{L}\in \bigoplus_{j>1} H^{4j}(P_0;\mathbb{Q})$ be a class such that
\[\tau_*(\mathcal{L}\cap [P_0])=0\in H_{4n-4*}(B\pi_1 (P_0) ;\mathbb{Q}).\]
Then there is some non-zero integer $R$ and a homotopy equivalence $f\colon P_1\to P_0$ of closed smooth manifolds such that the $L$-polynomials of $P_0$ and $P_1$ satisfy the equation
\[L(P_1) = f^* \bigl(L(P_0) + R\mathcal{L}\bigr).\]
\end{thm}

\begin{proof}[Sketch of proof] As we need a slight modification of the argument later, we sketch the proof.
For more details, proofs or references concerning some of the statements
below we refer the reader to \cite{Davis}.

At first we identify the set of topological normal invariants on $P_0$ with $[P_0, G/\Top]$, which we equip 
with  the Abelian group structure so that  the surgery obstruction is a group
homomorphism. There is an  isomorphism
\[ \ell\colon [P_0, G/\Top] \otimes \mathbb{Q} \xrightarrow{\cong} \bigoplus_{j>0} H^{4j}(P_0, \mathbb{Q})\]
sending the class of a degree one normal map $f \colon (P_1, TP_1) \to (P_0,\xi)$
to the class $\ell(f )=\frac18(L(\xi)L(P_0)^{-1} -1)$. Thus there exists a non-zero integer $R_1$ and a 
normal invariant $f \in [P_0, G/\Top]$ such that $\ell(f)=R_1 \mathcal{L} L(P_0)^{-1}$. We note that for this normal invariant  the desired
equation of $L$-polynomials for $L(P_1)=f^*L(\xi)$ holds with $8R_1$ instead of $R$. 

We need  to compute the surgery obstruction for this normal invariant $f$. 
The surgery obstruction 
\[[P_0, G/\Top] \otimes \Q  \to L_{4n}(\mathbb{Z} [\pi_1 (P_0) ])\otimes \mathbb{Q} \]
factors along
\begin{multline*} [P_0, G/\Top]\otimes \mathbb{Q} \xrightarrow[\ell]{\cong}
  \bigoplus_{k>0} H^{4k}(P_0; \mathbb{Q}) \xrightarrow{(-\cup L(P_0))\cap
    [P_0]}\\
 H_{4n-4*}( P_0; \mathbb{Q})\xrightarrow{\tau_*} H_{4n-4*}(B\pi_1
  (P_0);\mathbb{Q}).
\end{multline*}

Hence by assumption the surgery obstruction of $f$ 
 is zero in $L_{4n}(\Z [\pi_1 (P_0) ]) \otimes \Q$. This implies that there is a non-zero integer $R_2$  and 
 a normal invariant $f \in [P_0, G/\Top]$ so that $\ell(f) = R_2 R_1 \mathcal{L} 
 L(P_0)^{-1}$ and so that the surgery obstruction of this $f$ vanishes. Because 
$\pi_1(P_0)$ is either trivial or equal to $\Z$ in our case, we have $L_{4n}(\Z [\pi_1 ( P_0) ]) \cong \Z$ by the 
 Bass-Heller-Swan splitting theorem, so that we can in fact choose $R_2 =1$.

Finally it follows from \cite{Weinberger} that for each normal invariant $ f \in [P_0, G/ \Top]$ some multiple 
of $f$  lies in the image of the canonical map $[P_0 ,G / O] \to [P_0, G / \Top]$. Hence 
we find a non-zero multiple $R_3$ of $R_2R_1$ and  a smooth normal invariant $f \in [P_0 ,G/O]$ 
which satisfies $\ell(f) = R_3 \mathcal{L} L(P_0)^{-1}$  and 
whose surgery obstruction vanishes.  Performing surgery along this normal invariant 
yields a homotopy equivalence $f : P_1 \to  P_0$ so that the stated 
equation holds with $R = 8 R_3$. 
\end{proof}

We apply this result to the following situation: Let
\begin{equation*}
e_{k} \in H^{k}(S^{k} ;
\Z),  \quad e_{4\alpha-k} \in H^{4\alpha-k}(S^{4\alpha-k} ; \Z), \quad e_{4\beta}
\in H^{4\beta}(S^{4\beta} ; \Z)
\end{equation*}
be generators. As $P_0$ is stably
parallelizable, $L(P_0)=1$. Moreover, $\pi_1(P_0)$ is trivial if $k\geq  2$ or
infinite cyclic if $k=1$. We conclude from Theorem \ref{Davis_theorem} that we can find $R\ne 0$ and a homotopy
equivalence  
\[
    f \colon P_1 \to P_0 
\]
of smooth closed manifolds so that all homogeneous components of the
Hirzebruch $L$-class of $P_1$ vanish  except
\[
     L_{\alpha} = R \cdot (e_{k} \times e_{4\alpha-k}) , \quad L_{\beta} = R \cdot e_{4\beta}. 
\]
Here we use the identification
\[
f^* \colon H^*(S^{k} \times S^{4\alpha-k} \times S^{4\beta} ; \Q)  \cong H^*(P_1 ; \Q). 
\]

We will show that a manifold with the above properties automatically has
non-zero $\hat{A}$-genus. For $j \geq 1$ we  denote by $\lambda_j^L\in\rationals$ the coefficient of the 
degree $4j$ Pontryagin class $p_j$ in $L_j(p_1, ..., p_j)$, and by $\lambda^A_j$ the corresponding coefficient of the $\hat{A}$-polynomial $\hat{A}_j(p_1, \ldots, p_j)$.  

\begin{lem}\label{computation_of_coefficients}
We have
\begin{align*}
   \lambda_j^L & = \frac{2^{2j} (2^{2j-1}-1)}{(2j)!}\cdot B_j,\\
   \lambda_j^A & =-\frac{1}{2(2j)!}\cdot B_j
\end{align*}
where $B_j$ denotes the $j$-th Bernoulli number, defined by the relation (cf. \cite[p. 281]{Milnor})
\[
     \frac{x}{\tanh x} = 1 + \sum_{j = 1}^{\infty} (-1)^{j-1} \frac{B_j}{(2j)!}(2x)^{2j} .
\]
\end{lem}

\begin{proof}
According 
to \cite[Problem 19-C]{Milnor} we have equations 
\begin{eqnarray*}  
    L(t) \frac{d (t/ L(t))}{dt}  & = &  1  + \sum_{j=1}^{\infty} (-1)^j  \lambda^L_j \cdot t^j \\
    \hat{A}(t) \frac{d (t/ \hat{A}(t))}{dt}  & =  & 1  + \sum_{j=1}^{\infty} (-1)^j \lambda^{\hat{A}}_j \cdot t^j
\end{eqnarray*} 
where 
\[
    L(t)  =  \frac{\sqrt{t}}{\tanh(\sqrt{t})} \, , \hspace{0.5cm} \hat{A}(t)=  \frac{\frac{1}{2} \sqrt{t}}{\sinh(\frac{1}{2}\sqrt{t})}
\]
are the formal power series for the multiplicative sequences $\{L_n\}$ and $\{\hat{A}_n\}$.

For the $\hat{A}$-polynomial we compute
\[
     \hat{A}(t) \frac{d (t/ \hat{A}(t))}{dt} = \frac{1}{2} + \frac{1}{2} \frac{\frac{1}{2} \sqrt{t}}{\tanh (\frac{1}{2} \sqrt{t}) } \, . 
\]
This gives (setting $x = \frac{1}{2} \sqrt{t}$ in the relation defining the Bernoulli numbers) 
\[
     \hat{A}(t) \frac{d (t/ \hat{A}(t))}{dt} = 1 + \frac{1}{2} \cdot \sum_{j=1}^{\infty} (-1)^{j-1} \frac{B_j}{(2j)!}  t^j \, . 
\]
This implies the second equation. The first equation appears explicitly in \cite[Problem 19-C]{Milnor} or can be obtained by similar methods.
\end{proof}

This lemma implies that all $\lambda^L_j$ and $\lambda^A_j$ are different from zero. By our choice of $\beta=\alpha+1$ or $\beta=\alpha+2$, it follows recursively that all the Pontryagin classes of $P_1$ vanish except 
possibly $p_\alpha$, $p_\beta$, and $p_n$, and that $p_{\alpha}$ and $p_{\beta}$ as well as $p_{\alpha} \cdot p_{\beta}$ 
are non-zero. 

Let us write the degree $4n$-components of 
the $L$- and the $\hat{A}$-polynomial of an arbitrary vector bundle with vanishing $p_j$ for  $j \neq \alpha, \beta, n$  as
\begin{eqnarray*} 
  L_n(p_\alpha, p_\beta , p_n) &  = & \lambda^L_n \cdot p_n + \mu^L \cdot p_\alpha p_\beta \\
  \hat{A}_n (p_\alpha, p_\beta, p_n) & = & \lambda^{A}_n \cdot p_n + \mu^{A} \cdot p_\alpha p_\beta \, . 
\end{eqnarray*}
with rational numbers $\mu^L, \mu^{A}$. 

\begin{prop} \label{calc} We have 
\[
    \frac{\mu^L}{\lambda^L_n} \neq \frac{\mu^A}{\lambda^A_n} 
\]
and hence the following implication for a $4n$-dimensional connected oriented
manifold $M$ with $p_j(M) = 0$ for $j \neq \alpha,\beta, n$:
\[
 {\rm If} \quad p_{\alpha} (M) \cdot p_{\beta}(M)  \neq 0 \quad \text{and} \quad   L(M)
    = 0,   \quad {\rm then} \quad \hat{A}(M) \neq 0 \, . 
\]
\end{prop}

\begin{proof} This is a calculation in universal characteristic classes. To keep the notation transparent, assume 
that $E_\alpha,E_\beta$ are vector bundles with total Pontryagin classes $1 + p_\alpha$ and $1 +p_\beta$, in particular
  $L_n(E_\alpha)=0=L_n(E_\beta)$. By the multiplicativity of the total
  Pontryagin class and the universal $L$- and $\hat{A}$-polynomials we obtain 
  \begin{equation*}
       p_n(E_\alpha\oplus E_\beta) = p_\alpha\cdot  p_\beta,  \quad
  p_\alpha(E_\alpha\oplus E_\beta) = p_\alpha, \quad  p_\beta(E_\alpha\oplus 
  E_\beta) =  p_\beta .
  \end{equation*}
This implies 
\begin{eqnarray*}
  (\lambda^L_n+\mu^L)\cdot (p_\alpha \cdot  p_\beta) &=&  L_n(E_\alpha\oplus E_\beta)\\
 = 
  L_n(E_\alpha)+L_n(E_\beta) 
    +L_\alpha(E_\alpha)\cdot L_\beta(E_\beta) &=& \lambda^L_\alpha \cdot \lambda^l_\beta \cdot 
    (p_\alpha \cdot p_\beta), 
 \end{eqnarray*}
 and hence 
\[
    \lambda^L_n+\mu^L =  \lambda^L_\alpha \lambda^L_\beta.
\]
An analogous computation shows 
\[
    \lambda^{A}_n + \mu^{A}   =  \lambda^{A}_\alpha \lambda^{A}_\beta,  
\]
so that altogether we obtain
\[ 
  1+ \frac{\mu^L}{\lambda^L_n}  =  \frac{\lambda^L_\alpha
    \lambda^L_\beta}{\lambda^L_n}, \quad
  1+  \frac{\mu^A}{\lambda^A_n}  =  \frac{\lambda^A_\alpha
    \lambda^A_\beta}{\lambda^A_n}.
\]

From Lemma   \ref{computation_of_coefficients}  it follows that
\[ 
  1+  \frac{\mu^L}{\lambda^L_n} =  -
  \frac{2(2^{2\alpha-1}-1)(2^{2\beta-1}-1)}{2^{2n-1} - 1} 
  \frac{\lambda_\alpha^{A}\lambda_\beta^A}{\lambda^{A}_n } = -
  \frac{2(2^{2\alpha-1}-1)(2^{2\beta-1}-1)}{2^{2n-1} - 1}\left(
    \frac{\mu^A}{\lambda^A_n} +1\right) . 
\]
Because $ -\frac{2(2^{2\alpha-1}-1)(2^{2\beta-1}-1)}{2^{2n-1} - 1}  \ne
1$ the conclusion follows.
\end{proof}

If $k=1$ then the map
\[
    \phi_0 \circ f \colon P_1 \to S^{1} 
\]
is homotopic to the projection map of a smooth fiber bundle. This follows from
Farrell's obstruction theory over the circle \cite[Theorem
6.4]{FarrellIndiana}. Note 
that in the case at hand $\phi_0\circ f$ induces an isomorphism of fundamental
groups, so that the kernel of the $\pi_1$-homomorphism is the trivial group. By
\cite[Remarks on p.~316]{FarrellIndiana}, the fibering obstructions vanish and
one only has
to check that the universal covering of $P_1$ is homotopy equivalent to a
finite CW-complex. But this
is homotopy equivalent to
the universal covering of $P_0$ and therefore has this property.

If $k\geq 2$ the map $\phi_0\circ f$ will usually not be homotopic to a fiber
bundle projection and a more complicated construction is needed. The theory
which is relevant for the following discussion was developed by Casson
\cite{Casson} and Hatcher \cite{Hatcher}.

Recall that for any continuous map $f \colon X \to Y$ of topological spaces,
the {\em homotopy fiber $L$} of $f$ is the fiber of the map
\begin{eqnarray*} 
     E_f & \to & Y \\
     (x,\gamma) & \mapsto & \gamma(1) 
\end{eqnarray*} 
where 
\[
   E_f = \{ (x, \gamma) ~|~ x \in X, \gamma\colon [0,1] \to Y, \gamma(0) = f(x) \} \, . 
\]
Let $b \in S^m$ be the north pole, viewed as base point in $S^m$. 
Let $V$ be a  closed smooth manifold $V$ equipped with a smooth map $f\colon
V\to S^m$ such that the homotopy fiber is simply-connected. Casson
\cite[Section 1]{Casson}  defines such a map
to be a {\em pre-fibration} if the point $b
\in S^m$ is a regular value, $V\setminus f^{-1}(b)$ is simply-connected and
the canonical inclusion  
\[
       f^{-1}(b) \to L \, , ~ v \mapsto (v, \gamma\colon t \mapsto f(v)) 
\]
of the point-preimage into the homotopy fiber of $f$ is a weak homotopy equivalence. In this case the
smooth manifold $F = f^{-1}(b)$ is called the {\em fiber} of $(V,f)$.  

\begin{prop} Let $R \neq 0$ be the number appearing in the construction of $f : P_1 \to P_0$ 
after the proof of Theorem \ref{Davis_theorem}. If we construct the homotopy equivalence 
$f : P_1 \to P_0$ using the number $2R$ instead of $R$,
  then the map $\phi_0 \circ f
  \colon P_1 \to S^{k}$ constructed above is homotopic to a pre-fibration
  $\phi_1\colon P_1 \to S^{k}$.
\end{prop} 

\begin{proof} 
After applying a homotopy to $\phi_0\circ f$ we obtain a map $g$ for which the
value $b\in S^{4k}$ is regular with some fiber $F$. By \cite[Lemma 4]{Casson}
we can assume that $F$ and $P_1\setminus F$ are simply connected. By
\cite[Lemma 2 and the proof of Theorem 1]{Casson}
 the
inclusion from $F$ into the homotopy fiber of $g$ is a degree one normal map.
From \cite[Theorem 1 and page 497]{Casson} we know that the obstruction for
$g$ to be homotopic to a pre-fibration
is given by the (simply-connected) surgery obstruction of this degree one
normal map. The surgery obstruction groups in odd dimensions being trivial, we
have to distinguish between the cases where $k$ is divisible by 4 and where
$k$ is congruent $2$ modulo $4$.

In the case where $k$ is divisible by $4$, the surgery obstruction is an
integer which (up to a multiple) is given by the difference of signatures. The
homotopy fiber $L'$ of $g$ is homotopy equivalent to $S^{4\alpha-k}\times
S^{4\beta}$, so that its signature is zero. On the other hand, by Hirzebruch's
signature formula we obtain the signature of $F$ (which is framed in $P_1$) by 
\[\sigma(F)=\langle L(P_1), [F]\rangle\, ,\]
which is zero since by our choices of $\alpha$ and $\beta$ there is no $L$-class in the relevant degree. 

In the case where $k$ is congruent 2 modulo 4, the surgery obstruction group is cyclic of order 2. In the notation of the proof of Theorem \ref{Davis_theorem}, this surgery obstruction is given by the composite group homomorphism
\[[P_0, G/\Top]\to [S^{4\alpha-k}\times S^{4\beta}, G/\Top] \to L_{4n-k}(\mathbb{Z})\cong \mathbb{Z}/2\]
where the first map is given by restriction. It follows that this obstruction becomes zero if the number $R$ appearing in Theorem \ref{Davis_theorem} is multiplied by 2.
\end{proof}

In general the map $\phi_1$ is not homotopic to 
an actual (smooth) fiber bundle: By the theory developed in \cite{Casson} there is an 
obstruction lying in  a  homotopy group of a certain concordance space. Our goal is to show that $(P_1,\phi_1)$ can be replaced by some $(P_2, \phi_2)$ which fibers over $S^k$.

Casson \cite[Section 4]{Casson} calls  two pre-fibrations $f \colon V \to S^m$
and $ f'\colon V' \to S^m$ {\em equivalent} if there is a diffeomorphism $V
\to V'$ compatible with $f$ and $f'$ in a neighborhood of $f^{-1}(b)$. This
implies that the fibers of $f$ and $f'$ are diffeomorphic. Let
$F$ denote this common fiber. Following Casson, equivalence classes of pre-fibrations are classified as follows:
Let $\tilde{A}(F)$ be the simplicial group of block diffeomorphisms of $F$. Recall from \cite[p.~5]{Hatcher} that the 
$k$-simplices in $\tilde{A}(F)$ are given by diffeomorphisms of $F \times \Delta^k$ restricting to diffeomorphisms of 
$F \times \tau$ for each face of $\Delta^k$. An element in $\pi_k
(\tilde{A}(F), \id)$ is represented by a self-diffeomorphism of $F\times D^k$ which
is the identity in a neighborhood of $F\times S^{k-1}$, and two elements agree
if and only if they are concordant by a concordance which keeps a neighborhood
of $F\times S^{k-1}$ fixed. Thinking of $D^k$ as the northern hemisphere
$D^k_+\subset S^k$, such an $f$ extends by the identity map to a
self-diffeomorphism $f'$ of $F\times S^k$. We may use $f'$ to glue two copies
of $F\times D^{k+1}$ along their common boundaries to obtain a closed manifold
$V$ which comes with a canonical projection to $S^{k+1}$, collapsing $F\times
D^k_+$ to a point, which is easily seen
to be a pre-fibration which we call $V(x)$. 

{In the following we will suppress the obvious base
point $\id$ in the notation of homotopy groups.} 

\begin{prop} Let $F$ be simply connected and of dimension at least $6$, and let $k\geq 2$. Then the rule $x\mapsto V(x)$ defines a one-to-one correspondence 
between $\pi_{k-1} (\tilde{A}(F))$ and equivalence classes of pre-fibrations over $S^k$ with fiber $F$.
\end{prop} 

\begin{proof}
This is \cite[Lemma 6]{Casson}, once one has identified $\pi_{k}( \tilde{A}(F))$
with what Casson calls $D_{k}(F)$. The latter group is given by the
diffeomorphisms of $F\times S^k$ keeping a neighborhood of $F\times D^k_-$
pointwise fixed, modulo concordance keeping $F\times D^k_-$ pointwise
fixed. (Here $D^k_-$ is the lower hemisphere.) There is  a canonical map
$D_k(F)\to \pi_k ( \tilde{A}(F)) $ given by restriction of a diffeomorphism to
$F\times D^k_+$. Its inverse is given by extending a diffeomorphism of
$F\times D^k_+$ by the identity.
\end{proof}

For a pre-fibration $f \colon V \to S^k$ with fiber $F$, which we assume to be simply connected and of dimension at least $6$, 
we denote by $h(V,f) \in \pi_{k-1}( \tilde{A}(F)) $ the corresponding element. We call this the {\em characteristic element} of the pre-fibration. We now formulate a condition when a pre-fibration is equivalent to a fiber bundle. 
 
Let $A(F)$ be the simplicial group of diffeomorphisms of $F$, where 
the $k$-simplices are given by diffeomorphisms of $F \times \Delta^k$ which are compatible with the projection to $\Delta^k$. 
Note that $A(F)$ is a simplicial subgroup of $\tilde{A}(F)$, which is in general not normal.

If $x\in \pi_k ( \tilde{A}(F)) $ lies in the image of $\pi_k ( A(F)) $, then $V(x)$ is
obtained (up to equivalence) by gluing two copies of $V\times D^k$ along a
diffeomorphism $V\times S^{k-1}$ which acts  fiberwise over $S^{k-1}$. It follows
that the projection $V(x)\to S^k$ is a smooth fiber bundle. Using the exact
sequence 
\[\pi_k ( A(F)) \to \pi_k ( \tilde{A}(F) ) \xrightarrow{\psi} \pi_k (\tilde{A}(F), A(F))\]
we obtain:

\begin{prop} Let $f \colon V \to S^k$ be a pre-fibration with simply connected fiber $F$ of dimension at least $6$ with characteristic element $h(V,f)  \in \pi_{k-1}( \tilde{A}(F)) $. 
 If $\psi(h) = 0 \in \pi_{k-1}(\tilde{A}(F), A(F))$, then $f \colon V \to S^k$ is (as a pre-fibration) equivalent to a smooth fiber bundle. 
 \end{prop}

 For a closed smooth manifold $K$ we consider the pre-fibration 
 \[
    \phi_1 \times \id  \colon  P _1 \times K \to S^{k} \, \quad (p,x) \mapsto \phi_1(p) 
 \]
with fiber $F \times K$.  Then $\psi (h( \phi_1 \times \id))$ is the image of
$\psi(h(\phi_1))$ under the map  of homotopy groups induced by the map 
 \[
   (\widetilde{A}(F), A(F)) \to (\widetilde{A}(F \times K), A(F \times K)) 
\]
which sends a diffeomorphism $\omega\colon F \times \Delta^k \to F \times \Delta^k$ to 
$\omega \times \id\colon F \times \Delta^k \times K \to F \times  \Delta^k \times K $.

 \begin{prop} \label{tricky} For a closed smooth manifold $K$ with vanishing
   Euler characteristic $\chi(K)$ and $r>0$,  the induced map 
 \[
      \pi_r(\tilde{A}(F), A(F)) \to \pi_r(\tilde{A}(F \times K^{r}), A(F
      \times K^{r}))
 \]
 is equal to $0$. Here $K^r$ denotes the $r$-fold Cartesian product of $K$
 with itself.
 \end{prop}
 
 \begin{proof}  We consider the Hatcher spectral sequence \cite[Proposition
   2.1]{Hatcher} 
 \[
      E_{pq}^1= \pi_q(C(F \times I^p)) \Longrightarrow \pi_{p+q+1}(\widetilde{A}(F), A(F)) \, . 
\]
By \cite[Proposition in Appendix I on p.~18]{Hatcher}, multiplication 
 with $K$ induces the zero map 
 \[
      \pi_j(C(F \times I^p)) \to \pi_j(C(F \times I^p \times K)) 
 \]
 for each $j$, because $\chi(K) = 0$. It follows  that in the filtration of $\pi_r(\tilde{A}(F), A(F))$ 
 induced by the $E^{\infty}$-term of the Hatcher spectral sequence the map induced by the product
 with $K$ reduces the filtration degree of each element by one. This implies the assertion of Proposition \ref{tricky}. 
 \end{proof}
 
Thus let $K$ be a closed $\max(l,1)$-connected smooth manifold with $\hat A(K)\neq 0$ and $\chi(K)=0$. (For instance, we may use the construction at the beginning of this section for appropriate values of $k$, $\alpha$ and $\beta$). Applying Lemma \ref{tricky} to our previous construction we conclude that  the pre-fibration 
 \[
     \phi_1 \times \id \colon  P_1 \times K^{k-1}  \to S^{k} 
 \]
 is equivalent to a smooth fiber bundle $\phi_2 \colon P_2  \to S^{k}$. Because $\hat{A}(K) \neq 0$ we still have $\hat{A}(P_2) \neq 0$.

It remains to prove that we can assume that we have a smooth section $s\colon S^{k} \to P_2$ with trivial 
normal bundle. First notice that by choosing $\alpha$ and $l$ large enough, there is a smooth section $s\colon S^{k} \to P_2$, of which we would like to show that its normal bundle is trivial.

Again the case $k=1$ is the simplest one. In fact a real bundle over the
circle is trivial if and only if its first Stiefel-Whitney class vanishes, and
the first Stiefel-Whitney class of the normal bundle agrees with
$w_1(P_0)=0$.

In the case where $k\geq 2$ we argue as follows. Since we are in the stable range, the normal bundle is classified by an element in $\pi_k ( BO) $. If $k\equiv3,5,6,7$ modulo $8$, the group $\pi_k ( BO) $ is zero, so that the normal bundle is automatically trivial. In the cases where $k=4l$ is divisible by 4,  we have $\pi_k BO=\mathbb{Z}$ and non-trivial bundles over $S^k$ may be detected by their  $l$-th 
rational Pontryagin classes. The $l$-th rational Pontryagin class of the normal bundle is equal to $s^*(p_l(P_2)) \in H^{4l}(S^{4l} ; \Q)$. 
But by the above construction of $P_2$ this class is equal to $0$, so that the normal bundle is trivial in this case, too.

The remaining cases are $k\equiv 1$ or $2$ modulo $8$, in which case $\pi_k
( BO) =\mathbb{Z}/2$.  We do not see a general reason why the normal bundle should
be trivial in this case, but we describe a procedure how to change the
pre-fibration $\phi_1\colon P_1\to S^k$ so that the normal bundle of the
embedded $S^k$ in $P_1$ becomes trivial. Since $P_2$ is diffeomorphic to
$P_1\times K^{k-1}$, this will imply that the normal bundle of the embedding
$S^k\to P_2$ is also trivial.

Recall that by Casson's classification result, if $\dim F \geq 6$ and $F$ is simply conneced any pre-fibration over $S^k$ with fiber $F$ is equivalent to one of the form $V(x)$, with $x\in \pi_{k-1}(\tilde{A}(F)) $. If $F$ is $k$-connected and at least $(k+2)$-dimensional, then there is an embedding $S^k\to V(x)$ such that the composite map $S^k\to V(x)\to S^k$ is homotopic to the identity, and any two such embeddings are isotopic. We call such an embedding simply the embedding of $S^k$ into $V(x)$ and denote by $\nu_x\colon S^k\to BO$ the classifying map of its normal bundle.

\begin{lem}
Suppose that $F$ is $k$-connected and at least $(k+2)$-dimensional (where $k\geq 2$), and let $x,y\in \pi_{k-1}( \tilde{A}(F)) $. Then the normal bundle of the embedding of $S^k$ into $V(x+y)$ is classified by $\nu_x+\nu_y$. Moreover
\[\hat A(V(x+y))=\hat A(V(x)) + \hat A(V(y)).\]
\end{lem}

\begin{proof}
Let $x$ and $y$ be represented by automorphisms $\alpha$ and $\beta$ of
$F\times D^{k-1}$, fixing a neighborhood of $F\times S^{k-2}$, and denote by
$\alpha'$ and $\beta'$ the corresponding automorphisms of $F\times
S^{k-1}$. Because 
the element $x+y$ is
represented by the composition $\beta\circ\alpha$, the space $V(x+y)$ can be written as a ``fibered
connected sum''
\[
{V(x+y)\cong \left( F\times D^{k}_+\right)  \cup_{\alpha'} \left( F\times S^{k-1}\times [0,1] \right) \cup_{\beta'} \left( F\times D^{k}_- \right) } . 
\]

We use this identification to describe the embedding of $S^k$ into $V(x+y)$. To do that, let us first describe the embedding $i_x$ of $S^k$ into $V(x)$. Choose a base point $*\in F$. On the lower hemisphere $D^k_-$ the embedding $i_x$ is given by the inclusion $i_-\colon D^k_-\to D^k_-\times F$ at the base point. The composite
\[S^{k-1}\xrightarrow{i_-} S^{k-1}\times F \xrightarrow{\alpha'} S^k\times F\]
extends continuously to a map $D^k_+\to D^k_+\times F$ (coning off and using a
null homotopy 
of the map ${\rm pr}_F\circ \alpha'\circ i_-\colon S^k\to F$). The extension may be
approximated by a smooth
embedding $i_+$ in such a way that $i_+$ and $i_-$ together define the
embedding of $S^k$ into $V(x)$.

Similarly one obtains a choice of the embedding $i_y$ of $S^k$ into $V(y)$
which is the inclusion at the base point on the upper hemisphere. Denoting the
embedding at the base point by $j\colon S^{k-1}\times [0,1]\to F\times
S^{k-1}\times [0,1]$, it follows that 
\begin{multline*}
  i_x\vert_{D^k_+}\cup j\cup i_y\vert_{D^k_-}\colon\\
    D^k_+\cup \left( S^{k-1}\times [0,1]\right) \cup D^k_-\to \left( F\times
      D^{k}_+\right) \cup_{\alpha'} \left( F\times S^{k-1}\times [0,1] \right)
    \cup_{\beta'} \left( F\times D^{k}_- \right)
\end{multline*}
defines the embedding of $S^k$ into $V(x+y)$.

The normal bundle of this embedding is given by $\nu_x\vert_{D^k_+}$ on
$D^k_+$ and by $\nu_y\vert_{D^k_-}$ on $D^k_-$, while on $S^{k-1}\times [0,1]$
the normal bundle is trivialized. It follows that the classifying map
$\nu_{x+y}$ factors as
\[
S^k\to S^k\vee S^k \xrightarrow{\nu_x\vee \nu_y} BO
\]
where the first map is the pinch map. But this composite defines the sum $\nu_x+\nu_y$ in the homotopy group $\pi_k (BO)$ so that the first statement follows.

Finally notice that $V(x+y)$ is obtained from $V(x)\amalg V(y)$ by cutting out
two copies of $F\times D^k$ and gluing along the resulting $F\times S^{k-1}$
boundaries, defining a parametrized connected sum. Hence there is a standard bordism
between these $V(x+y)$ and the disjoint union 
$V(x)\cup V(y)$. This implies the second statement.
\end{proof}

Suppose now that in our situation, the normal bundle of $S^k$ in $P_1=V(x)$ is
non-trivial. As $\pi_k (BO)$ has order 2, replacing $P_1$ by $P'_1=V(2x)$
yields a pre-fibration with the same fiber $F$, whose section has trivial normal
bundle and such that $2\hat A(P_1)=\hat A(P'_1)\neq 0$. This completes the proof of Theorem
\ref{technical} in the case $k \geq 1$.

If $k = 0$ we can use the same construction as for case $k \geq 1$, but starting with $P_0 = S^{4 \alpha} \times S^{4 \beta}$. 
With Theorem \ref{Davis_theorem}, Lemma \ref{computation_of_coefficients}
and Proposition \ref{calc} we find an $N(0,l)$ so that for all $n \geq N(0,l)$ there is a $4n$-dimensional $l$-connected closed spin 
manifold $F$ with $\hat{A}(F) \neq 0$. Now we define $E$ as the disjoint union of two copies of the spin manifold $F$. 
The surjective map to $S^0$ is constant on each component.

\end{document}